\documentclass[a4paper,11pt]{article}
\usepackage{mathtext}         %

\usepackage{latexsym}
\usepackage{amsmath,amssymb,amsthm,amsopn,amscd}
\usepackage[dvips]{graphicx}
\usepackage{array,fullpage,wrapfig}

\usepackage[matrix,arrow,curve]{xy}
\usepackage[active]{srcltx}

\usepackage{euscript}

\newcommand{\C}{{\mathbb{C}}}   
\newcommand{\Z}{{\mathbb{Z}}}   

\newcommand\g{{\mathbf{g}}}
\newcommand\La{{\mathfrak{g}}}  
    
\newcommand\h{{\mathfrak{h}}}     
\newcommand{\n}{{\mathfrak{n}}}   
\newcommand{\p}{{\mathfrak{p}}} 
\newcommand\calF{{\mathcal F}}  
\newcommand\calO{{\mathcal O}}  
\newcommand{\calM}{{\mathcal{M}}}

\newcommand\QQ{{\mathsf{Q}}}    
\newcommand\PP{{\mathsf{P}}}    
\newcommand\PPdom{{\mathsf{P}_{\Delta}}} 
\newcommand\Proj{{\EuScript{P}}}

\def\sl{{\mathfrak{sl}}}        

\newcommand\gl{{\mathfrak{gl}}}  

\newcommand{\Cat}{{\mathcal{C}}}    
\newcommand{\amod}{{\mathcal{I}}}   
\newcommand{\famod}{{{\mathcal{I}}}}   
\newcommand{\KG}{\mathsf{K}_{0}}           

\newcommand{\rr}{\mathsf{r}}  
\newcommand{\ww}{\mathsf{r}^{!}}  

\newcommand{\bi}{{\bar{\imath}}}

\newcommand{\oDelta}{\overline{\Delta}} 
\newcommand{\onabla}{\overline{\nabla}} 

\newcommand{\calD}{{\mathcal{D}}}

\newcommand{\ldot}{{\:\raisebox{1.5pt}{\selectfont\text{\circle*{1.5}}}}}
\newcommand{\udot}{{\:\raisebox{4pt}{\selectfont\text{\circle*{1.5}}}}}

\newcommand{\ttt}{\text{-}}

\let\le\leqslant

\let\leq\leqslant
\let\geq\geqslant
\let\cong\simeq

\newcommand{\dual}{\checkmark}  

\DeclareMathOperator*{\ooplus}{{\oplus}}
\DeclareMathOperator*{\ootimes}{{\otimes}}

\newtheorem{theorem}[equation]{Theorem}
\newtheorem*{theorem*}{Theorem}
\newtheorem{proposition}[equation]{Proposition}
\newtheorem*{proposition*}{Proposition}

\newtheorem{lemma}[equation]{Lemma}
\newtheorem{corollary}[equation]{Corollary}
\newtheorem*{corollary*}{Corollary}
\newtheorem{conjecture}[equation]{Conjecture}

\theoremstyle{definition}
\newtheorem{example}[equation]{Example}
\newtheorem{definition}[equation]{Definition}

\theoremstyle{remark}
\newtheorem{remark}[equation]{Remark}
\title{
Highest weight categories and Macdonald polynomials
}
\author{
Anton Khoroshkin
\thanks{
A.~Khoroshkin:
\newline 
Laboratory of Mathematical Physics \& Faculty of Mathematics, Vavilova 7,
\newline
 National Research University Higher School of Economics, Moscow 117312, Russia
\newline
and
Institute of Theoretical and Experimental Physics (ITEP Moscow),
\newline
Bolshaya Cheremushkinskaya 25, Moscow 117259, Russia
\newline
akhoroshkin@hse.ru
}
}
\date{}
\begin{document}

\maketitle

\begin{center} \emph{Dedicated to my teacher Boris Feigin for his 60's birthday} \end{center} 

\begin{abstract}
The aim of this paper is to introduce the categorical setup
which helps us to relate the theory of Macdonald polynomials and
the theory of Weyl modules for current Lie algebras discovered by V.\,Chari and collaborators.

We identify Macdonald pairing with the homological pairing on the ring of characters of the Lie algebra of currents $\mathbf{g}\otimes\C[x,\xi]$.
We use this description in order to define complexes of modules whose Euler characteristic of characters coincide with Macdonald polynomials.
We generalize this result to the case of graded Lie algebras with anti-involution.
We show that whenever the BGG reciprocity holds for the corresponding category of modules then these complexes collapse 
to the modules concentrated in homological degree $0$.
The latter modules generalizes the notion of Weyl modules for current Lie algebras and the notion of Verma modules in the 
BGG category $\mathcal{O}$.
We give different criterions of BGG reciprocity and apply them to the Lie algebra of currents $\mathbf{g}\otimes\mathbb{C}[x]$
with $\mathbf{g}$ semisimple.
\end{abstract}

\setcounter{section}{-1}

\section{Introduction}
\label{sec::intro}
Let $\g$ be a semisimple Lie algebra with the weight lattice $\PP$ and the Weyl group $W$.
Then the ring of symmetric functions $\Z[\PP]^{W}$ is the ring of characters of finite-dimension $\g$-modules
with monomial basis $m_\lambda=\sum_{\alpha\in W\lambda} e^{\alpha}$ given by summation over the orbits 
where $\lambda$ runs over the set of integral dominant weights $\PPdom$.
I.\,Macdonald defined a two-parametric pairing $\langle,\rangle_{q,t}$ on $\Z[\PP]^{W}$ (see e.g.~\cite{Macdonald1} and references therein).
Macdonald polynomials $\{P_{\lambda}(q,t),\lambda\in\PPdom\}$ form another basis in the ring of symmetric functions $\Z[\PP]^{W}$ 
which is orthogonal with respect to Macdonald pairing and transformation matrix from 
Macdonald basis to monomial basis is upper-triangular with respect to the following standard partial ordering on weights:
$$
\lambda \geq \mu\ \stackrel{def}{\Longleftrightarrow} \lambda-\mu = \sum_{\alpha \text{ is simple}} d_{\alpha} \alpha \text{ with } d_{\alpha}\geq 0.
$$
In other words, Macdonald polynomials are uniquely defined by the following properties:
$$
\langle P_\lambda, P_\mu\rangle_{q,t} = 0 \text{ if } \lambda\neq \mu \quad \& \quad P_{\lambda} = m_{\lambda} +\sum_{\mu<\lambda} c_{\lambda,\mu} m_{\mu}
$$

One of the main purposes of this paper is to give a possible categorification of Macdonald polynomials
while considering the category of bigraded modules over the Lie superalgebra of currents $\g\otimes\C[x,\xi]$. 
Here the first $q$-grading is assigned to the even variable $x$
and the second $t$-grading is assigned to the odd variable $\xi$.
We denote by $\famod=\famod(\g\otimes\C[x,\xi])$ the category of bi-graded finitely generated modules  
over $\g\otimes\C[x,\xi]$ which are locally finite dimensional
 with respect to the action of the semisimple subalgebra $\g = \g\otimes 1$.
The Lie ideal $\g\otimes\C[x,\xi]_{+}$ of currents with zero constant term is nilpotent and, therefore, the irreducible graded $\g\otimes\C[x,\xi]$ 
modules are just irreducible $\g$-modules with additional bi-grading. 
The corresponding Grothendieck ring is isomorphic to the ring of generating series in two variables of $\g$-characters of $\famod(\g\otimes\C[x,\xi])$ 
and, thus, is identified with  $\Z[\PP]^{W}[[q,t,q^{-1},t^{-1}]]$.
There exists a canonical anti-involution $\tau$ on the Lie algebra $\g\otimes\C[x,\xi]$ which extends the Cartan anti-involution on $\g$ and fixes $x$ and $\xi$.
For a bi-graded module $N=\oplus N_{i,j}$ we denote by $N^{\dual}$ the module that is graded linear dual to $N$ and the action is 
twisted by the anti-involution $\tau$; denote by $N\{k,l\}$ the module with the grading shifted $N\{k,l\}_{i,j}:=N_{i+k,j+l}$.
The key observation which is not very widely known  
is that the homological pairing on the Grothendieck ring (see e.g.~\cite{FGT},\cite{Fe1}):
\begin{equation}
\label{eq::hom::pair::intro}
\langle M, N \rangle_{\g\otimes\C[x,\xi]}:=
\sum_{\begin{smallmatrix}
i\geq 0, \\ 
k,l\in\Z
\end{smallmatrix}
} (-1)^{i} q^{k} t^{l} \dim Ext^{i}_{\g\otimes\C[x,\xi]\ttt mod}(M\{k,l\}, N^{\dual})
\end{equation}
given as a graded Euler characteristic of derived homomorphisms between $M$ and $N^{\dual}$ 
coincides with Macdonald pairing of corresponding graded characters.
It is worth mentioning that we use the description of derived homomorphisms via relative Lie algebra cohomology
$$
\oplus_{k,l} Ext^{\udot}_{\g\otimes\C[x,\xi]\ttt mod}(M\{k,l\}, N^{\dual})  = H^{\udot}(\g\otimes\C[x,\xi],\g; Hom_{\C}(M,N^{\dual}) ).
$$
It is natural to ask if one can realize Macdonald polynomials as graded $\g$-characters of modules in $\famod(\g\otimes\C[x,\xi])$.
Unfortunately, Macdonald polynomials are not Schur positive and, therefore, may not be characters of true modules.
In the case of $\g=\gl_n$ one makes an appropriate change of variables to get the Schur positivity. 
This is the famous result on positivity of modified Macdonald polynomials proved by M.\,Haiman~\cite{Haiman}.
We do not want to make any change of variables and, moreover, we want to have a uniform result for all semi-simple Lie algebras.
Therefore, we have to work with negative coefficients in the Schur basis decomposition.
Thus, the only possibility we have is to pass to the derived setting.
That is, one expects complexes whose Euler characteristic of graded characters will coincides with Macdonald polynomials.

First, we refine the partial ordering on integral dominant weights to a total linear ordering.
Second, for each dominant weight $\lambda$ we denote by $\famod^{\leq\lambda}$ 
the full subcategory of modules from $\famod$ whose weight decomposition with 
respect to the action of Cartan subalgebra in $\g$ contains only weights which are less than or equal to $\lambda$.
Let $\calD^{-}(\famod(\g\otimes\C[x,\xi]))$ be the derived category of  
 complexes of $\g\otimes\C[x,\xi]$-modules from $\famod$ whose homological degrees are bounded from above and
 for all pairs $k,l\in\Z$ the sum of dimensions of graded components
with the inner bi-degree $(k,l)$ is finite.
We denote by $\calD_{\leq\lambda}^{-}(\famod)$ its full subcategory of complexes whose homology belongs to $\famod^{\leq\lambda}$
and $\bi_{\lambda}:\calD_{\leq\lambda}^{-}(\famod) \rightarrow \calD^{-}(\famod)$ the corresponding inclusion functor.
One of our main results is the following

\begin{theorem}{(Theorem~\ref{thm::Mac::Der::gen})}
\label{thm::Mac::derived} 
There exists a left adjoint functor $\bi^{!}_{\lambda}: \calD^{-}(\famod) \rightarrow \calD_{\leq\lambda}^{-}(\famod)$.
The objects $\bi_{\lambda}(\bi_{\lambda}^{!} \Proj(\lambda))$ form an orthogonal basis with respect 
to the homological pairing~\eqref{eq::hom::pair::intro}:
$$
Ext_{\g\otimes\C[x,\xi]\ttt mod}^{\udot}\left(\bi_{\lambda}(\bi_{\lambda}^{!} \Proj(\lambda))\{k,l\}, \bi_{\mu}(\bi_{\mu}^{!} \Proj(\mu))^{\dual}\right) =
\begin{cases}
0, \text{ if } \lambda\neq \mu, \\
[H^{\udot}(\bi_{\lambda}(\bi_{\lambda}^{!} (\Proj(\lambda) ))_{k,l}]^{\lambda}
,
 \text{ if } \lambda=\mu
\end{cases}
$$
Here $L(\lambda)$ is the irreducible $\g$-module of highest weight $\lambda$, also considered as an irreducible ${\g\otimes\C[x,\xi]}$ 
whose bigraded component are zero except bidegree $(0,0)$. 
Its projective cover $Ind_{\g}^{\g\otimes\C[x,\xi]}L(\lambda)$ is denoted by $\Proj(\lambda)$.
The notation $[M^{\udot}_{k,l}]^{\lambda}$ stands for the subspace of weight $\lambda$ in a $\g$-module $M_{k,l}$.
\end{theorem}
The immediate corollary of this statement gives a categorification of Macdonald polynomials:
\begin{corollary*}
 The dual Macdonald polynomial of weight $\lambda$ is the graded Euler characteristic of $\bi_{\lambda}(\bi_{\lambda}^{!} \Proj(\lambda))$.
Respectively, the Macdonald polynomial of weight $\lambda$ is the fraction of 
a graded Euler characteristic of $\bi_{\lambda}(\bi_{\lambda}^{!} \Proj(\lambda))$
by the generating series of Euler characteristics of multiplicities 
$z_{\lambda}(q,t):=\sum_{i,j} q^{i}(-t)^{j} [\bi_{\lambda}(\bi_{\lambda}^{!} \Proj(\lambda))_{i,j}: L(\lambda)]$.
The scalar product of the Macdonald polynomial of weight $\lambda$ with itself is equal to $z_{\lambda}(q,t)^{-1}$.
\end{corollary*}

Let us notice that, unfortunately, we do not give below any reasonable description of functors $\bi_{\lambda}^{!}$
and hope to discuss them elsewhere.
Nevertheless, we point out that Theorem~\ref{thm::Mac::derived} has a straightforward generalization for any $\Z_{+}$-graded Lie algebra 
$\La=\La_0\oplus\La_1\oplus\ldots$ with semisimple $\La_0$, finite-dimensional $\La_{i}$ for all $i>0$ and graded anti-involution.
That is, we choose a decomposition $\La = \La^{+}\oplus\La^{0}\oplus\La^{-}$ generalizing the Cartan decomposition of $\La_0$;
 we consider the category $\famod(\La)$ of finitely generated graded $\La$-modules with locally finite action of $\La_0$;
define its subcategory $\famod(\La)^{\leq\lambda}$ of modules with upper bound on weights;
use the adjoint functor $\bi_{\lambda}^{!}:\calD^{-}(\famod(\La)) \rightarrow \calD_{\leq\lambda}^{-}(\famod(\La))$ to 
show that the image in $K$-group of complexes $\bi_{\lambda}(\bi_{\lambda}^{!} (Ind_{\La_0}^{\La} L(\lambda)))$ 
form an orthogonalization of 
the monomial basis with respect to the 
pairing on symmetric functions $\Z[\PP]^{W}[[q,q^{-1}]$ given via relative Lie algebra cohomology of $\La$:
\begin{equation}
\label{eq::hom::pair::gen::intro}
\langle M, N \rangle_{\La} := \sum_{i} (-1)^{i} dim_{q} H^{i}(\La,\La_0; Hom_{\C}(M,N^{\dual})) 
\end{equation}
Note that $Ind_{\La_0}^{\La} L(\lambda)$ is the minimal projective cover of irreducible $\La_0$-module $L(\lambda)$ of highest weight $\lambda$
and thus is further denoted by $\Proj(\lambda)$.

A reasonable attempt to understand the complexes $\bi_{\lambda}(\bi_{\lambda}^{!} (\Proj(\lambda)))$ is to describe those 
situations when the aforementioned complexes are true modules concentrated in homological degree zero.
For example, if the triangulated category $\calD_{\leq\lambda}^{-}(\famod) $ is equivalent to the triangulated category $\calD^{-}(\famod^{\leq\lambda})$ then
the adjoint functor $\bi_{\lambda}$ is the left derived functor of the corresponding 
left adjoint functor to the fully faithful inclusion $\imath_{\lambda}:\famod\rightarrow\famod^{\leq\lambda}$ of the abelian categories.
In this case the derived image of a projective module coincides with the nonderived one.
In particular, the derived image of a projective module is a projective module concentrated in homological degree zero.
That is 
$$\bi_{\lambda}^{!} (\Proj(\lambda)) =\imath_{\lambda}^{!} \Proj(\lambda) = \Delta(\lambda),$$
where $\Delta(\lambda)$ is a projective cover of $L(\lambda)$ in $\famod^{\leq\lambda}$.
The modules $\Delta(\lambda)$ are the universal highest weight modules and have an explicit description in terms of generators and relations.
The module $\Delta(\lambda)$ is the quotient of the module $\Proj(\lambda)=Ind_{\La_0}^{\La}L(\lambda)$ by additional relations $\La^{+}p_{\lambda}=0$
where $p_\lambda$ is a highest weight vector of $L(\lambda)$.
When the Lie algebra $\La$ is the algebra of currents $\g\otimes A$ with $\g$-semisimple and 
$A$ a commutative unital graded algebra (that is $A=\C\oplus A_1\oplus A_2\oplus\ldots$) the modules $\Delta(\lambda)$ were extensively studied by 
V.\,Chari and collaborators (see e.g.~\cite{CPweyl,FL,CFK}). 
In these works and in many subsequent papers on related subject the modules $\Delta(\lambda)$ are 
denoted by $W(\lambda)$ or $W_{\lambda}(A)$ and called (global) Weyl modules.
We prefer to call them standard modules following the ideology of BGG categories after~\cite{CPS}.
Moreover, we also use the categorical approach to the local Weyl modules $W_{\lambda}^{loc}(A)$ introduced in~\cite{FL}, 
which we call proper standard modules and denote them by $\oDelta(\lambda)$.
\emph{The proper standard module} $\oDelta(\lambda)$ is the projective cover of $L(\lambda)$ in the subcategory of $\famod^{\leq\lambda}$
consisting of modules that have $L(\lambda)$ with multiplicity $1$.
What is nice about local and global Weyl modules is that they have a nice description in terms of generators and relations 
(see Section~\ref{sec::stand::mod} below), however, their characters may be very complicated.

Thus, in order to be able to relate the theory of Weyl modules and Macdonald polynomials 
we have to give criterions to have an equivalence between 
triangulated categories $\calD_{\leq\lambda}(\famod)$ and $\calD(\famod^{\leq\lambda})$.
This is our next result (which slightly generalizes the known results for BGG categories~\cite{CPS}):
\begin{theorem*}{(Theorem~\ref{thm::def::strat})}
The following conditions on an abelian category $\famod(\La)$ are equivalent.
 \begin{enumerate}
 \item[(s1)] The triangulated categories
$\calD_{\leq\lambda}(\famod)$ and $\calD(\famod^{\leq\lambda})$ are equivalent.
\item[(s2)] The BGG reciprocity holds in $\famod$.
That is, for all dominant weights $\lambda$
the kernel of the projection $\Proj(\lambda)\twoheadrightarrow \Delta(\lambda)$ admits a filtration whose successive quotients 
are isomorphic to $\Delta(\mu)$ with $\mu>\lambda$. 
 \item[(s3)] The second extension group $Ext_{\Cat}^{2}(\Delta(\lambda),(\oDelta(\mu))^{*})=0$ vanishes 
for all pairs of dominant weights $\lambda,\mu$.
 \item[(s4)] For all pairs of dominant weights $\lambda,\mu$ we have the following vanishing conditions
\begin{equation}
 \dim Ext_{\Cat}^{i}(\Delta(\lambda),(\oDelta(\mu))^{*}) = \left\{
\begin{array}{l}
 1, \text{ if } \lambda=-\omega_0(\mu) \ \& \ i=0, \\
0, \text{ otherwise. }
\end{array}
\right.,
\end{equation}
where $\omega_0$ is the longest element of the Weyl group of $\La_0$ and $*$ stands for the graded linear dual module.
\end{enumerate}
\end{theorem*}
The result is very general and uses the formalism of highest weight categories (see Section~\ref{sec::famod::hwc}) that is roughly speaking 
is given by partial ordering on weights.
We are mostly interested in a particular case of modules over a graded Lie algebra $\La$.
Moreover if $\La$ admits an anti-involution then $\oDelta(\lambda)^{\dual} = (\oDelta(-\omega_0(\lambda)))^{*}$.
Consequently, if one of the above properties holds then the collection of graded characters of proper standard modules 
$\oDelta(\lambda),\lambda\in\PPdom$, forms an orthogonalization of the monomial basis with respect 
to the pairing~\eqref{eq::hom::pair::gen::intro} and refined partial ordering on weights.
Moreover the converse statement is also true:
\begin{theorem*}(Theorem~\ref{thm::MacWeyl::strat})
If the characters of proper standard modules $\oDelta(\lambda)$ form an orthogonal basis with respect to the pairing~\eqref{eq::hom::pair::gen::intro}
 and for all $\lambda$ the character of standard module $\Delta(\lambda)$ factors through the character of proper standard module 
then the BGG reciprocity holds in $\famod$.
\end{theorem*}
Another criterion of BGG reciprocity uses Lie algebra cohomology and the aforementioned factorization property:
\begin{proposition*}(Proposition~\ref{cor::famod::strat})
The BGG reciprocity holds for the category $\famod(\La)$ if and only if the following conditions are satisfied:
\\
$\phantom{......}$ 
$(\imath)$ 
the relative Lie algebra cohomology with trivial coefficients are almost trivial:
$$
H^{\udot}(\La,\La_0;\C) = H^{\udot}(\frac{\La^{0}}{\La^{0}\cap[\La^{+},\La^{-}]};\C)
$$
\\
$\phantom{......}$ 
$(\imath\imath)$ 
For all dominant $\lambda$ the standard module $\Delta(\lambda)$ is the tensor product of its graded subspace of weight $\lambda$ and 
the proper standard module $\oDelta(\lambda)$
\\
$\phantom{......}$ 
$(\imath\imath\imath)$ 
The generating series of the weight subspace of weight $\lambda$ in the module $\Delta(\lambda)$ is equal to the inverse of 
corresponding Macdonald constant term $\langle P_{\lambda},P_{\lambda}\rangle_{\La}$.
\end{proposition*}
The knowledge of the appropriate Lie algebra cohomology, couple of known results from the theory of Weyl modules and their characters 
we show that 
\begin{corollary*}{(Theorem~\ref{thm::g_A::stratified})}
The BGG reciprocity holds for the category $\famod(\g\otimes A)$ 
with $\g$ semisimple and a nontrivial graded super-commutative algebra
$A=\C\oplus A_1\oplus A_2\oplus \ldots$ implies that $A$ is the polynomial ring $\C[x]$ in one even variable.
The converse is true for all simply-laced $\g$ and conjecturally true for all semi-simple $\g$. 
In particular, for the Lie algebra $\g\otimes\C[x]$ the characters of the corresponding proper standard modules $\oDelta(\lambda)$ 
are given by specialization of Macdonald
polynomials at $t=0$.
\end{corollary*}
We believe that the category $\g\otimes\C[x]$ is stratified for all $\g$ and there exists a proof of this statement which does not use
the knowledge of characters of proper standard modules.

\emph{\bf{A historical remark.}}
 The latter orthogonal polynomials are called $q$-Hermite polynomials.
The result about characters of standard modules for $\g\otimes\C[x]$ was known for $\g$ simply laced (\cite{FL}) inspired by the fact that the Lie algebra $\g\otimes\C[x]$ 
is a parabolic subalgebra of affine Lie algebra $\hat{\g}$
and that corresponding level one Demazure modules coincide with Weyl modules in simply laced case.
However, it is known that in non-simply laced case Demazure modules are smaller.
The BGG reciprocity for $\g\otimes\C[x]$ was first proven for $\g=\sl_2$ in~\cite{BCM} and, later on, for $\g=\sl_n$ in~\cite{BBCKL}.
However, both proofs use the information on characters of Weyl modules for $\g=\sl_n$ 
and, therefore, can not be generalized to the non simply laced situation where the characters of Weyl modules is not yet known.
While preparing this manuscript to the publishing another paper by V.\,Chari and B.\,Ion was appeared on arXiv~\cite{CI} where the
authors announced that they verified the coincidence of characters of local Weyl modules and $q$-Hermite operators and deduce from that
the BGG reciprocity using the same methods as in our Theorem~\ref{thm::MacWeyl::strat}.

The BGG reciprocity is such a nice property that it is worth mentioning nice applications of this theory following
the philosophy of~\cite{CPS}.
In particular, one defines excellent filtrations and tilting modules just from the general categorical setup.
We are interested in particular applications for one-dimensional currents, considered as 
parabolic subalgebras of affine Kac--Moody Lie algebras. 
The latter theory recovers the theory of excellent filtrations on integrable representations of positive level
recently studied by I.\,Cherednik and B.\,Feigin in~\cite{Chered:Feigin} on the level of characters and 
previously inspired by A.\,Joseph (e.g.~\cite{Joseph} and references therein) as filtrations by Demazure modules.
We do not add something new to the general theory of tilting modules discussed in~\cite{Donkin}, 
however, for the case of current Lie algebras we recover the results of~\cite{BC}.

We want to finish the introduction by mentioning that all results of the paper are adapted to the case of representations 
of graded Lie (super)algebras, 
where one can relate the Ext-pairing with Macdonald pairing.
However, one can deduce the same description of orthogonal basis and the BGG reciprocity in a much more general setup.
The representations of associative algebras like Yangians will be of a particular interest.

\subsection{Outline of the paper}
The paper is organized in the opposite order compared to the introduction where we wanted to motivate our categorical conclusions as much as possible.

That is, in Section~\ref{sec::framework} we describe the initial data.
Namely, we say what we require from a nonsemisimple Lie algebra (Subsection~\ref{sec::nonsem::Lie}) 
in order to define its category of modules $\famod$,
such that the corresponding homological pairing generalizes the Macdonald one (Subsection~\ref{sec::pairing}).
In Subsection~\ref{sec::list::examples} we give a list of test examples.
Generalizing the categorical approach of~\cite{CFK} we define analogues of local and global Weyl modules in 
Section~\ref{sec::weyl::modules}.

Section~\ref{sec::bgg::hwc} is concerned with the general theory and applications of categories with BGG reciprocity 
(the so-called \emph{stratified categories}).
In Subsection~\ref{sec::famod::hwc} we explain what kind of properties from $\famod$ are important to state the 
general theory of \emph{highest weight categories} which is presented in the subsequent Subsection~\ref{sec::hwcat}.
We state and prove the main criterion of BGG categories in Subsection~\ref{sec::strat}
and use it in all the examples in Section~\ref{sec::examples}.
We recall applications to the theory of excellent filtrations and tilting modules in Subsection~\ref{sec::BGG::appl}.

We show that BGG reciprocity holds in some hand-made examples (Subsection~\ref{sec::ex::zero}) 
and in the case of one-dimensional currents $\g\otimes\C[x]$ (Subsection~\ref{sec::ex::C_x_}).
However, we show that it is not true for finite-dimensional parabolic subalgebras (Subsection~\ref{sec::ex::parab}) and
for $\g\otimes A$ with $A\not\cong\C[x]$ (Subsection~\ref{sec::ex::cur}).

The last Section~\ref{sec::gen::Mac} is devoted to the description of Macdonald polynomials as characters of certain objects of derived categories
in the case when a highest weight category $\famod$ does not admit a BGG reciprocity.

\subsection{Acknowledgment}
I am grateful to A.\,Bondal, I.\,Cherednik, B.\,Feigin, A.\,Kirillov Jr., S.\,Loktev, I.\,Losev, N.\,Markarian, M.\,Movshev 
for stimulating discussions, comments and useful references.

This work is partially supported by "The National Research University--Higher School of Economics" Academic Fund Program in 2014-2015, 
research grant 14-01-0124.
The Author thanks Isaac Newton Institute at Cambridge and Simons Center for Geometry and Physics at Stony Brook 
for hospitality and perfect working conditions that allowed 
the author to prepare this manuscript.

The initial draft of this paper appeared online on my homepage in May 2013, and I would like to thank all my colleagues
who bothered me to finally post it onto arXiv.
\section{Framework}
\label{sec::framework}
\subsection{Non semi-simple Lie algebras with anti-involution}
\label{sec::nonsem::Lie}
Everything works well for an algebraically closed field $\Bbbk$ of zero characteristic and to simplify the notation
we suggest to work over complex numbers $\C$. 

Let us fix a $\Z_{+}$-graded Lie algebra
with finite-dimensional graded components such that 
the zero'th component $\La_0$ is a semi-simple Lie algebra:
\begin{equation}
\label{def::Lie_grading}
\begin{array}{c}
\La:= \La_0\oplus \La_1\oplus\La_2\ldots, \quad [\La_{i},\La_{j}] \subset \La_{i+j}, \\
\La_0 \text{ is semisimple and } \dim(\La_i)<\infty.
\end{array}
\end{equation}
We assume that the Lie algebra $\La$ is finitely generated as a Lie algebra as it is always the case in all interesting examples.

Denote by $\amod=\amod(\La)$ the category of finite-dimensional $\Z$-graded modules.
Respectively by $\famod^{gr}=\famod(\La)$ we denote the category of finitely generated $\Z$-graded modules where the action of $\La_0$
is locally finite. 
That is a module $M\in\famod^{gr}$ is a graded module with finite-dimensional graded components such that grading is bounded in negative degrees:
$$
\famod(\La):= \{ M=\oplus_{i\in\Z}^{\infty} M_i\  :\  
\La_{i}\cdot M_j \subset M_{i+j},\ \forall i\in\Z\ \dim M_i < \infty,\ \ \text{ and }  \ M_{i}=0 \  \text{ for } i\ll 0\}
$$
It is natural to require morphisms in $\famod$ to preserve grading.
However, the shift of the inner grading defines an auto-equivalence of the category $\famod$.
We define a module $M\{d\}$ to be a module $M$ with the grading shifted by $d$:
$$
M = \oplus_{i\geq i_0} M_i \quad \Longrightarrow M\{d\} := \oplus_{i\geq i_0-d} M\{d\}_{i}, \text{ with } M\{d\}_{i}:= M_{d+i}
$$
The category $\famod$ is obviously abelian and as we will see in the next Section~\ref{sec::simples} has enough projectives and 
explicit description of simples and their projective covers.
\begin{remark}
Our main example is the bigraded Lie algebra of currents $\g\otimes\C[x,\xi] = \oplus_{i\geq0,\ j\in\{0,1\}} \g\otimes x^{i} \xi^{j}$.
Thus, in this case it is reasonable to consider the category of finitely generated bigraded modules $M=\oplus_{i,j\in\Z} M_{i,j}$.

Consequently, if the Lie algebra $\La$ is multigraded $\La=\oplus_{v\in\Z_{+}^{r}}\La_{v}$ one has to replace
the category $\famod(\La)$ with the category of multigraded finitely generated $\La$-modules with finite-dimensional graded components.
All further conclusions of this article remains to be true in this situation.
\end{remark}

\begin{remark}
We will use the lower index for the grading and reserve the upper index for the weight decomposition.
\end{remark}

We denote by $\famod(\La)^{op}$ the opposite category of graded $\La$-modules with finite-dimensional graded components 
bounded from above. 
$$ \famod(\La)^{op}:= \left\{ M\in\La\ttt mod : M =\oplus_{i\in\Z} M_i, \text{ with } \dim M_i <\infty, 
\text{ and } M_i=0 \text{ for }i \gg 0\right\} $$
 Modules in this category are not finitely generated but are expected to be finitely cogenerated.
Moreover, $M\in\famod$ if and only if the graded linear dual representation $Hom_{\C}(M,\C) =\oplus_{i}Hom_{\C}(M_{-i},\C)$ 
belongs to $\famod^{op}$.
As we will see later on in Section~\ref{sec::simples} the abelian category $\famod$ has enough projectives and, respectively,
 the category $\famod^{op}$ has enough injectives.
A graded $\La$-module $M$ belongs to both $\famod$ and $\famod^{{op}}$ if and only if its dimension is finite.

In order to define symmetric Macdonald pairing we underline the case when the Lie algebra $\La$ admits
 an anti-involution $\tau$ such that it preserves the grading on $\La$
and its restriction on $\La_{0}$ coincides with Cartan anti-involution:
$$
\tau:\La_{i} \rightarrow \La_{i} \quad \text{ and } \quad \forall f,g\in\La \ \ \ [\tau(f),\tau(g)]=-\tau([f,g]).
$$
Using this feature we can assign another duality functor $\dual: \famod\rightarrow \famod^{op}$
 which will not change the $\La_0$-character of a module.
We say that the \emph{dual} module $M^{\dual}$ of $M\in\famod$ is isomorphic to its linear graded dual
as a vector space and the action of $\La$ is twisted by anti-involution $\tau$:
$$
\forall g\in \La, f\in Hom(M,\C)   \quad  g\cdot f(m):= f(\tau(g)\cdot m). 
$$
The notion of a dual module $M^{\dual}$ is well defined for any module in $\famod$.
However, if a module $M$ is not finite-dimensional then the module $M^{\dual}$
will have infinitely many negative nonzero degrees  and will not belong to $\famod$.

\subsection{A list of examples of graded Lie algebras}
\label{sec::list::examples}
Let us give a list of examples of $\Z_{+}$-graded Lie algebras enhanced with an anti-involution.
We hope that this will increase the readers interest to this subject.
\begin{enumerate}

\item 
\label{ex::parabolic}
Let us show that parabolic subalgebra of a Kac-Moody Lie algebra admits a natural $\Z_{+}$-grading. 
Indeed, let $\g$ be a Kac-Moody Lie algebra of rank $r$.
Let $\Pi=\{\alpha_1,\ldots,\alpha_r\}$ be the set of simple roots of $\g$;
$\{e_\alpha,h_\alpha,f_\alpha\}$ stands for the standard basis of the $\sl_2$-triple associated with a simple root $\alpha$. 
Any given subset $S$ of $\Pi$ defines a parabolic subalgebra $\p_{S}$ generated by $\{f_\alpha|\alpha\in S\}$ and $\{e_\beta| \beta\in \Pi\}$.
Suppose that Kac-Moody subalgebra generated by $\{f_\alpha,e_\alpha | \alpha\in I \}$ is finite-dimensional.
Let us set the degree of $e_\alpha$ and $f_{\alpha}$ be equal to zero iff $\alpha\in S$ and the degree of $e_\beta$ for $\beta\in (\Pi\setminus S)$
be equal to $1$. 
It is a straightforward check that this uniquely defines a grading 
on the parabolic subalgebra $\p_{S}$, which obviously satisfies conditions~\eqref{def::Lie_grading}.

\item 
\label{eq::example::parabolic}
We want to underline one precise case of the latter situation:
$$
\g\otimes\C[[t]] \hookrightarrow \hat{\g} = \widehat{\g\otimes\C[[t,t^{-1}]}.
$$
Namely, the Ivahori parabolic subalgebra $\g\otimes\C[[t]]$ of an affine Lie algebra $\hat{\g}$
has an evident grading with respect to $t$ grading and Cartan anti-involution of $\g$ produces an anti-involution of the entire Lie algebra of currents.

\item
\label{ex::currents}
Generalizing the previous example we can associate
the graded Lie algebra $\g\otimes A$ to any commutative $\Z_{+}$-graded algebra $A=\C\oplus A_1\oplus A_2\oplus\ldots$ 
and a semisimple Lie algebra $\g$. 
The grading on $\g\otimes A$ comes from the grading on $A$. 
The corresponding representation theory $\famod(\g\otimes A)$ was initiated by Chari and Pressley in~\cite{CPweyl}
and leads to the theory of Weyl modules.

\item 
\label{ex::Wn}
The Lie algebra $L_0(n)$ of formal vector fields on $\C^{n}$ vanishing at the origin:
$$
L_0(n):= \left\{\sum_{i=1}^{n} f_i\frac{\partial}{\partial x_i} \ : \ f_i\in\C[[x_1,\ldots,x_n]], f_i(0)=0\right\}
$$
The grading of $x_i^{k}\frac{\partial}{\partial x_j}$ is equal to $k-1$.
This algebra does not admit any anti-involution. 
The component of degree $0$ is isomorphic to $\gl_n$, so if one wants to apply to it further conclusions of this manuscript it is better to mod out
$L_0$ by its one-dimensional center $\sum_{i=1}^{n} x_i\frac{\partial}{\partial x_i}$

\item
\label{ex::hamilton}
The Lie subalgebra of Hamiltonian vector fields with trivial constant and linear terms:
\begin{gather*}
{\cal H}_{n}:= \{ f\in \C[[x_1,\ldots,x_{2n}]] : f(0)=\frac{\partial f}{\partial x_i}(0)= 0 \ \forall i\} \\
[f,g]:= \sum_{i=1}^{n} \left(\frac{\partial f}{\partial x_{2i-1}}\frac{\partial g}{\partial x_{2i}} 
- \frac{\partial g}{\partial x_{2i-1}}\frac{\partial f}{\partial x_{2i}}\right).
\end{gather*}
The grading is given by the grading on polynomials  decreased by $2$,
the anti-involution $\tau$ interchanges variables $x_{2i-1}$ and $x_{2i}$;
and the semisimple subalgebra consisting of quadratic functions is isomorphic to the symplectic Lie algebra $\mathfrak{sp}(2n,\C)$.

\end{enumerate}

\subsection{Generalizing Cartan decomposition and refining partial ordering on weights.}
\label{sec::notation}
Let us choose a Cartan decomposition of the semisimple part of the entire Lie algebra $\La$:
$$\La_{0}=\n^{-}\oplus\h\oplus\n^{+}$$
where $\h$ is a Cartan subalgebra and $\Phi$ is the associated root system.
We fix a choice of a set of positive simple roots 
 $\Pi=\{\alpha_1,\ldots,\alpha_r\}\subset\h^{*}$,
such that ($\oplus_{\alpha\in \Phi_{+}} \La_{0}^{\alpha}$) and 
($\oplus_{\alpha\in \Phi_{-}} \La_{0}^{\alpha}$) are the weight decomposition of $\n^{+}$ and $\n^{-}$ respectively.
Following $\sl_2$ case we denote by $\{e_i,h_i,f_i\}$ the $\sl_2$ triple assigned to a simple root $\alpha_i$,
such that $\forall h\in\h$  $[h,e_i]=\alpha_i(h) e_i$ and $\alpha_i$ is positive.
Denote by $\QQ$ the root lattice, that is an abelian group generated by $\Pi$;
and by $\QQ_{+}:=\{\alpha\in \QQ: \alpha =\sum m_i \alpha_i,\ m_i\geq 0\}$ we denote the cone of positive roots.
(We use the convention that $0$ belongs to $\QQ_{+}$.)
Notation $\PP$ stands for the weight lattice, which is generated by the set of fundamental weights
$\omega_1,\ldots,\omega_r$.
Let $\PPdom$ be the subset of integral dominant weights, that is the set of linear combinations of 
fundamental weights with nonnegative coefficients.
Recall, that  elements of $\PPdom$
are in one-to-one correspondence with the set of irreducible finite-dimensional representations of 
semisimple Lie algebra $\La_{0}$.
The root lattice $\QQ$ is a sublattice of the weight lattice $\PP$ of finite index.

By the weight decomposition of a $\La_{0}$-module  we mean its decomposition with respect to the action of Cartan subalgebra $\h\subset\La_{0}$.
Indeed, each finite-dimensional $\La_{0}$-module $M$ admits a weight decomposition 
$$M=\oplus_{\lambda\in \PP} M^{\lambda} \quad\text{where } M^{\lambda}:=\{m\in M \ : \ \forall h\in \h\ h\cdot m = \lambda(h) m\}. $$ 
We set {\it the weight support $wt(M)$} to be the subset 
$\{\lambda\in \PP\ : \ M^{\lambda}\neq 0\}$ 
of all integral weights $P$. 
In particular, any $\La$-module $M\in\famod$ and each graded component $M_k$ of this module admits 
a weight decomposition, since $M_k$ is a finite-dimensional $\La_{0}$-module.
Here and below whenever talking about weights we mean $\h$-action which is semisimple in all our discussions. 

Let us fix a nonzero nonintegral vector $h_0\in\h$ that separates positive and negative roots:
$$\forall\alpha\in\Pi \text{ we have } \langle\alpha, h_0\rangle >0.$$
Let us say that the weight $\lambda\in\PP$ is positive iff $\langle\lambda,h_0\rangle >0$,
respectively negative if $\langle\lambda,h_0\rangle <0$.
Therefore, as soon as $h_0$ is fixed 
each finite-dimensional $\La_0$-module $V$ has a decomposition 
 $V^{+}\oplus V^{0} \oplus V^{-}$
consisting of vectors of positive, zero and negative weights respectively.
In particular, any graded component $\La_i$ is a finite-dimensional $\La_{0}$-module and we get an extension of the Cartan decomposition:
\begin{equation}
\label{eq::gen::cartan}
\La_{i} = \La_{i}^{+}\oplus\La_{i}^{0}\oplus\La_{i}^{-} \text{ and }\La_{>0} = \La_{>0}^{+}\oplus\La_{>0}^{0}\oplus\La_{>0}^{-} 
\text{ with }  \La^{+} =\oplus_{i\geq 0} \La_{i}^{+},\ \La^{0} =\oplus_{i\geq 0} \La_{i}^{0}, 
\La^{-} =\oplus_{i\geq 0} \La_{i}^{-}\oplus\La_{i}^{0}
\end{equation}
Note that all three subspaces $\La_{>0}^{+}$, $\La_{>0}^{0}$ and $\La_{>0}^{-}$ form a subalgebra of $\La_{>0}$.
We extend the set of positive integral weights in the convex cone of simple roots with respect to the decomposition~\eqref{eq::gen::cartan}:
$$\PP_{+}(\La):= \left\{\sum_{i\geq 0} k_i\lambda_i\ :\ k_i\in\Z_{+},\ \lambda_i\in \La_{i}^{+}\right\}.
$$
Note that $\PP_{+}(\La)$ contains all simple roots because they are presented as weights of $\La_0^{+}=\n^{+}$.
Using this we weaken the partial ordering on the set of integral dominant weights:
\begin{equation}
\label{eq::def::ordering}
\lambda \geq_{\La} \mu \stackrel{def}{\Longleftrightarrow} \lambda -\mu \in \PP_{+}(\La).
\end{equation}
The transitivity of this partial ordering is obvious and the antisymmetry follows from the fact that $h_0$ is chosen not to be integral,
and, in particular, for an integral weight $\lambda\in\PP$ vanishing of the scalar product $\langle \lambda, h_0 \rangle =0$ 
implies that $\lambda=0$.

\subsection{Simples and projectives}
\label{sec::simples}
As soon as notations are fixed we are able do describe the set of irreducible graded $\La$-modules.
We summarize the description of simple and projective objects in $\famod$ in the following Lemma~\ref{lem::proj} below.

Recall that irreducible finite-dimensional $\La_{0}$-modules are numbered by their highest weights which are dominant.
For any $\lambda\in\PPdom$ we denote by $L(\lambda)$ the corresponding irreducible $\La_{0}$-module with highest weight $\lambda$.
The grading of any module in the category $\famod$ may be shifted with respect to the inner grading and we denote this shift operation using
figure brackets in order to underline the discrepancy with homological shift:
$$
\forall M=\oplus_{i\geq i_0}M_i \quad M\{k\}:=\oplus_{i\geq k+i_0} M\{k\}_{i} \ \text{ with }\ M\{k\}_{i}  := M_{k+i}
$$
Let $L(\lambda,k):=L(\lambda)\{-k\}$ be the irreducible graded $\La$-module with the trivial action of radical $\La_{>0}$
and with the only nonzero graded component placed in degree $k$.

\begin{lemma}
\label{lem::proj}
The set $\{L(\lambda,k) |\lambda\in \PPdom, k\in\Z\}$ form the set of irreducibles up to isomorphisms in the category $\famod$.

The induced representation $Ind_{\La_{0}}^{\La} L(\lambda)\{k\}:= U(\La)\otimes_{U(\La_{0})} L(\lambda,k)$ 
is the projective cover $\Proj(\lambda,k)$ of the irreducible module $L(\lambda,k)$. 
The grading on the induced module above comes from the grading on $\La$.
\end{lemma}
\begin{proof}
Each module $M\in\famod$ is bounded from below and hence admits a surjection of degree $0$ on the module $L(\lambda,-k)$ where 
$k=min\{{i\in\Z}\ :\ \dim M_{i}\neq 0\}$. This implies that each module $M\in\famod$ admits a filtration whose successive quotients are 
isomorphic to $L(\lambda,k)$ with $\lambda\in\PPdom$ and $k\in \Z$. 

Recall that by projective cover of a module $L$ we mean a minimal projective module that surjects onto $L$.
First, there is a natural surjection $Ind_{\La_{0}}^{\La} L(\lambda)\twoheadrightarrow L(\lambda)$ which sends the image of elements
of the radical  $\La_{>0}$ to zero.
Second, we have  the following identity for derived 
homomorphisms valid for all $\La$-modules $M$.
$$
Ext_{\La\ttt mod}^{i}(Ind_{\La_{0}}^{\La}L(\lambda,k),M)= Ext_{\La_{0}\ttt mod}^{i}(L(\lambda,k),M) =
\begin{cases}
{0, \text{ if } i>0  }\\
Hom_{\La_0}(L(\lambda), M_{-k}), \text{ if } i=0
\end{cases}
$$
Each module $M$ from $\famod$  decomposes into a direct sum of finite-dimensional $\La_{0}$-modules 
and, therefore, all higher homomorphisms vanish. In particular, $\forall M\in\famod$ we have
$Ext^{1}_{\famod}(Ind_{\La_{0}}^{\La} L(\lambda),M)=0$ what is equivalent to being a projective module.
Third, the irreducible module $L(\lambda)$ has a cyclic vector $v_{\lambda}$
which remains to be a cyclic vector for the induced module $Ind_{\La_{0}}^{\La} L(\lambda)$.
Consequently, the induced module is indecomposable and, therefore, a projective cover of $L(\lambda,k)$.
\end{proof}

\subsection{Grothendieck ring and Macdonald pairing}
\label{sec::pairing}
The category $\famod(\La)$ is abelian and we can describe the Grothendieck ring of this category:
$$
\KG(\famod):= \frac{\Z[ M | M\in \famod]}{ [M]-[N]+[K]=0, \text{ whenever } 0\to M \to N \to K \to 0 }.
$$ 
Lemma~\ref{lem::proj} explains that irreducibles in $\famod$ 
are the same as in the category of finite-dimensional modules over semisimple Lie algebra 
$\La_{0}$ with outer grading added.
Therefore, the image of a module $M$ in the Grothendieck ring $\KG(\famod)$ is uniquely defined by its graded $\La_{0}$-character:
$$
\chi_{q\ttt\La_{0}}(M):= \sum_{i\in \Z} q^{i} \left( \sum_{\lambda\in \PP} (\dim M_{i}^{\lambda})e^{\lambda}\right).
$$
We get the following isomorphism:
$$
\KG(\famod) \cong\KG(\text{fin.dim.} \La_{0}\ttt mod)[[q,q^{-1}] \cong \Z[\PP]^{W}\otimes\Z[[q,q^{-1}].
$$
Where parameter $q$ stands for the $\Z$-grading on the module and $W$ is the Weyl group of $\La_{0}$.
Modules in $\famod$ has zero graded component for sufficiently large negative integers and that is why we have to consider
Loran series in $q$ which are polynomial in $q^{-1}$ and power series in $q$.
Respectively, the Grothendieck ring $\KG(\famod^{op})$ 
will be the ring if symmetric functions with coefficients in formal power series in $q,q^{-1}$
 that are polynomial in $q$ and may have series in $q^{-1}$.
Following standard notations the elements of the linear basis of the set of integral weights $\Z[\PP]$ 
are written in the exponential form $e^{\lambda}$ with $\lambda\in\PP$.
The ring of characters of $\La_{0}$ will be referred as a ring of symmetric functions.

Let us define a bilinear map $\langle,\rangle_{\La}: \KG(\famod) \times \KG(\famod^{op}) \rightarrow \Z[[q,q^{-1}]$
 in the following natural way: 
\begin{equation}
\label{eq::ext::pair::nonsym}
\langle M , N\rangle_{\La}:= \sum_{i\geq 0} (-1)^{i} \sum_{k\in\Z}q^{k}\dim Ext^{i}_{\famod}(M\{k\},N) = 
 \sum_{i\geq 0} (-1)^{i} \dim_{q^{-1}} H^{i}(\La,\La_{0};Hom_{\C}(M,N)).  
\end{equation}
where $\dim_{q^{-1}}$ stands for the graded dimension with respect to the inner grading on the Lie algebra $\g$, 
that is $\dim_{q^{-1}}(H)=\sum_{j\in Z} q^{-j}\dim H_{j}$.
The notation $H^{\udot}(\La,\La_{0};K)$ is used for the relative Lie algebra cohomology of the Lie algebra $\La$ 
relative to its semisimple subalgebra $\La_{0}$
and with coefficients in $\La$-module $K$.
Note, that if the Lie algebra $\La$ admits an anti-involution $\tau$ 
then one can define the symmetric pairing on $\KG(\famod)$ using the duality functor:
\begin{equation}
\label{eq::ext::paring}
\begin{array}{c}
{ \langle,\rangle_{\La}: \KG(\famod) \times \KG(\famod) \rightarrow \Z[[q,q^{-1}]  }\\ 
\langle M , N\rangle_{\La}:= \sum_{i\geq 0} (-1)^{i} \dim_{q^{-1}} H^{i}(\La,\La_{0};Hom_{\C}(M,N^{\dual})).
\end{array}
\end{equation}
We will refer the pairing~\eqref{eq::ext::paring} as \emph{generalized Macdonald pairing} 
and we will motivate our notation in Remark~\ref{rem::Macd::pair} below.
\begin{lemma}
The generalized Macdonald pairing~\eqref{eq::ext::paring} is well defined symmetric nondegenerate pairing on the Grothendieck ring $\KG(\famod)$.
\end{lemma}
\begin{proof}
This statement seems to be known for specialists and therefore we will only sketch its proof.
Moreover, we will recall another description 
of this pairing in Lemma~\ref{lem::pair::equal} below.

The coincidence of derived homomorphisms and relative Lie algebra cohomology 
follows from the universal description of the Lie algebra cohomology 
as the derived functor (see e.g.~\cite{Weibel} for details).
One has to use the relative Lie algebra cohomology because all $\La_0$ modules are locally integrable (locally finite-dimensional).
Recall, that $Ext$-functor is the universal derived functor of homomorphisms which is derived with respect to both arguments.
 Thus, it maps short exact sequences to long exact sequences and, consequently, 
whenever Euler characteristics is defined it factors through the relations in the Grothendieck group.
Therefore, in order to show that the pairing is well defined it remains to show any kind of convergence 
of the Euler characteristics of derived homomorphisms.
Indeed, let us notice that for all $M\in\famod$ and $N\in\famod^{op}$ the module 
$Hom_{\C}(M,N)$ belongs to $\famod^{op}$, because its graded component of degree $k$ is equal to $\oplus_{j-i=k} Hom_{\C}(M_{i},N_{j})$
and, consequently, is finite-dimensional and even is zero for $k\gg 0$.
The graded component of inner degree $l$ of the Chevalley-Eilenberg complex 
$C^{\udot}(\La;K) = Hom_{\C}(\Lambda^{\udot} \La,K)$ (which computes the Lie algebra cohomology)
is isomorphic to the following sum
$$
\bigoplus_{p-\sum s_i = l, s_i\geq 0} Hom_{\C} (\Lambda^{\udot}(\La_{0})\bigotimes_{i} \Lambda^{s_i}\La_{s_i}, K_{p})
$$
and, hence, is finite-dimensional whenever $K$ belongs to $\famod^{op}$. 
Consequently, the subcomplex of relative cochains $H^{\udot}(\La,\La_0;Hom_{\C}(M,N))$ has also finite-dimensional graded components.

The symmetry property of the pairing follows from the observation that duality is an auto-equivalence in the category of graded $\La$-modules.
Namely, for all pairs of graded $\La$-modules $M,N$ we have an isomorphism
$$
Ext^{\udot}_{\La\ttt mod}(M\{k\},N^{\dual}) \cong Ext^{\udot}_{\La\ttt mod}(N\{k\},M^{\dual}).
$$

Finally, we notice that the set of irreducible $\La_{0}$-modules $\{L(\lambda,0) : \lambda\in\PPdom\}$ 
and the set of their projective covers $\{\Proj(\lambda,0) : \lambda\in\PPdom\}$ form a dual basis with respect to this pairing:
$$
\dim Ext^{i}_{\famod}(\Proj(\lambda,0)\{k\},L(\mu,0)) = \delta_{i,0}\cdot\delta_{k,0}\cdot \delta_{\lambda,\mu}
$$
Consequently, the pairing is nondegenerate. 
\end{proof}

The following Lemma gives the description of the pairing in terms of rings of characters
and gives a bridge to the Macdonald pairing on the ring of symmetric functions.
\begin{lemma}
\label{lem::pair::equal}
The pairing~\eqref{eq::ext::paring} induces the following pairing on the ring of characters $\Z[\PP]^{W}\otimes\Z[[q,q^{-1}]$:
\begin{equation}
\label{eq::pair::polyn}
\langle f, g \rangle_{\La} = \left[
f \cdot \overline{g} \cdot 
\prod_{\alpha\in \Phi}(1-e^{\alpha}) \prod_{r>0} \prod_{\lambda\in wt(\La_{r})} (1- q^{r} e^{\lambda})^{\dim \La_{r}^{\lambda}}
\right]_{1}
\end{equation}
with $f,g$ being a pair of symmetric polynomials from $\Z[\PP]^{W}\otimes\Z[[q,q^{-1}]$.
Here $\overline{g(e^{\alpha},q)}:= g(e^{-\alpha},q)$ and index $1$ near squared brackets means the following operation:
expand all brackets and take the coefficient near $e^{0}=1$ that is a Loran series in $q$ and is called the constant term.
\end{lemma}
\begin{proof}
The same result is known for the pairing on symmetric functions without any parameters
(see e.g.~\cite{Macdonald1}).
Recall that $\La_{0}$ is a semisimple Lie algebra associated with a root system $\Phi$.
Let $M$ be a finite-dimensional  $\La_{0}$-module with a character $f\in \Z[\PP]^{W}$ then 
the dimensions of the space of $\La_{0}$-invariants in $M$ has a description as a constant term in the sense of the coefficient near $e^{0}=1$:
\begin{equation}
\label{eq::const::term}
\dim [M]^{\g} = \left[ \prod_{\alpha\in \Phi} (1-e^{\alpha}) \cdot f \right]_{1}.
\end{equation}
Moreover, the graded $\La_{0}$-character of the dual representation
$M^{\dual}$ replaces $q$ by $q^{-1}$ and keeps the weight decomposition.
However, in the graded $\La_{0}$ character of the module $Hom(M,\C)$ 
both things are interchanged:
$e^{\lambda}$ by $e^{-\lambda}$ and  $q\mapsto q^{-1}$.
Consequently, the graded $\La_{0}$-character of the module $Hom(M,N^{\dual})$ is equal to 
$f\cdot \overline{g}$ with $q$ replaced by $q^{-1}$ in both functions.
It remains to point out that the right hand side of Equation~\eqref{eq::pair::polyn} is nothing but
the graded Euler characteristics of the relative Chevalley-Eilenberg complex with $q$ being replaced by $q^{-1}$:
\begin{multline*}
\sum_{i\geq 0}(-1)^{i}\dim_{q^{-1}}H^{i}(\La,\La_{0};Hom_{\C}(M,N^{\dual})) = 
\sum_{i\geq 0}(-1)^{i}\dim_{q^{-1}}C^{i}(\La,\La_{0};Hom_{\C}(M,N^{\dual}))) =\\
=\sum_{i\geq 0}(-1)^{i}\dim_{q^{-1}}Hom_{\La_0}\left(\Lambda^{i}(\La/\La_{0}), Hom_{\C}(M,N^{\dual})\right) = \\
=\sum_{i\geq 0}(-1)^{i}\dim_{q^{-1}}Hom_{\La_{0}}\left(\Lambda^{i}(\La_{>0}), Hom_{\C}(M,N^{\dual})\right) = \\
=\sum_{i\geq 0}(-1)^{i}\dim_{q^{-1}}\left[Hom\left(\Lambda^{i}\left(\ooplus_{\lambda\in \PP} \La_{>0}^{\lambda}\right),
 Hom_{\C}(M,N^{\dual})\right)\right]^{\La_{0}} = \\
=
\left[
\prod_{\alpha\in \Phi}(1-e^{\alpha}) \left(\prod_{r>0} \prod_{\lambda\in wt(\La_{r})} (1- q^{r} e^{\lambda})^{\dim \La_{r}^{\lambda}} \right)
f\overline{g}\right]_{1}.
\end{multline*}
The last equality is the application of the equality~\eqref{eq::const::term}.
\end{proof}
\begin{remark}
 Assume that $\La=\La_0\oplus\La_1\oplus\ldots$ is a graded Lie superagebra with $\La_0$-semisimple Lie algebra
and superdecomposition $\La_r={{\overline{\La}}}_r\oplus{\overline{\overline{\La}}}_r$ with  ${{\overline{\La}}}_r$ being even part
and super-part is denoted by ${\overline{\overline{\La}}}_r$. 
In this case the polynomial description of the pairing given in Equation~\eqref{eq::pair::polyn} has to be changed by appropriate denominators
coming from the super-part:  
\begin{equation}
\label{eq::pair::polyn::super}
 \langle f, g \rangle_{\La} = \left[
f \cdot \overline{g} \cdot 
\prod_{\alpha\in \Phi}(1-e^{\alpha}) \prod_{r>0} \prod_{\lambda\in wt(\La_{r})} 
\frac{(1- q^{r} e^{\lambda})^{\dim {{\overline{\La}}}_{r}^{\lambda}}}{(1- q^{r} e^{\lambda})^{\dim {\overline{\overline{\La}}}_{r}^{\lambda}}}
\right]_{1}
\end{equation}
This happens because the Chevalley-Eilenberg complex of a Lie superalgebra 
is no more the Grassman algebra but is the product of Grassman and symmetric algebras:
$$
C^{\ldot}_{CE}(\La,\La_0;K) = Hom_{\La_0}\left(\Lambda^{\udot}({{\overline{\La}}_{>0}})\otimes S^{\udot}({\overline{\overline{\La}}}),K\right).
$$
\end{remark}

Consider the free super-commutative polynomial algebra $\C[x,\xi]$ with one odd generator $\xi$ and one even generator $x$.
Following Example~\ref{ex::currents} with any semisimple Lie algebra $\g$ we can assign the graded Lie superalgebra $\g\otimes \C[x,\xi]$ 
with an anti-involution coming from Cartan anti-involution.
Denote by $\famod(\g\otimes\C[x,\xi])$ the category of finitely generated bi-graded modules over $\g\otimes\C[x,\xi]$.
\begin{proposition}
\label{rem::Macd::pair}
The two-parametric pairing coming from the $Ext$-pairing~\ref{eq::ext::paring} on the category of bi-graded modules $\famod(\g\otimes\C[x,\xi])$
coincides with the Macdonald pairing (\cite{Mac}) on the ring of symmetric functions:
$$
\langle f, g\rangle_{q,t}:=\left(\frac{\prod_{r>0}(1-q^{r})}{\prod_{r\geq 0}(1-t q^{r})} \right)^{rk(\g)} \cdot
\left[f\cdot \bar{g}  
\prod_{\begin{smallmatrix}
        r\geq 0, \\ \alpha\in\Phi
       \end{smallmatrix}
}\frac{1-q^{r} e^{\alpha}}{1-t q^{r} e^{\alpha}}\right]_{1}
$$
Parameters $q$ and $t$ corresponds to the two canonical gradings on modules and on the Lie algebra.
The parameter $q$ counts the degree with respect to even variable $x$
and parameter $t$ comes from the $\Z$-grading with respect to the odd variable $\xi$.
\end{proposition}
\begin{proof}
The weight decomposition of the adjoint representation is numbered by roots $\g=\h\oplus_{\alpha\in\Phi} \g^{\alpha}$ with $dim(\g^{\alpha})=1$.
Consequently, we have the following weight decomposition of $\g\otimes\C[x,\xi]$:
$$
\g\otimes\C[x,\xi]\cong \h\otimes\C[x]\oplus \h\otimes\C[x]\xi \ooplus_{r\geq 0} \left(\ooplus_{\alpha\in\Phi} 
\g^{\alpha}\otimes x^{r} \oplus \g^{\alpha}\otimes x^{r}\xi\right)
$$
It remains to substitute the weight decomposition into the Formula~\eqref{eq::pair::polyn::super} remembering 
 that the summands containing odd variable $\xi$ are odd. 
\end{proof}

This remark leads to the following definition of the generalized Macdonald polynomials:
\begin{definition}
\label{def::poly::Macdonald}
 The symmetric polynomials $P_{\lambda}\in \Z[\PP]^{W}[[q,q^{-1}]$ are called \emph{generalized Macdonald polynomials}
if they satisfy the two following properties:
\begin{itemize}
 \item $P_{\lambda} = m_\lambda +\sum_{\mu<_{\La}\lambda} c_{\lambda,\mu}(q) m_\mu$, 
where $m_\lambda:= \sum_{\alpha\in W\lambda} e^{\alpha}$ is the monomial basis in $\Z[\PP]^{W}$,
\item Polynomials $P_{\lambda}$ form an orthogonal basis with respect to the pairing coming from representations of $\La$:
$$
\langle P_{\lambda}, P_\mu\rangle_{\La} = 0 ,\ \ \text{ if } \lambda\neq \mu.
$$
\end{itemize}
\end{definition}
In other words, the basis of generalized Macdonald polynomials is the orthogonalization of the monomial basis with respect to the partial ordering on weights.
Instead of monomial basis one can also orthogonalize the Schur functions $s_\lambda$ given by the character of irreducible representation $L(\lambda)$.

\section{Highest weight modules in $\famod(\La)$}
\label{sec::weyl::modules}
\subsection{Highest weights}
Let us adapt the theory of highest weight vectors 
and representations to the category $\famod$.
Recall, that for a module $M$ we denote by $M_k^{\lambda}$ the weight subspace of $\h$-weight $\lambda$ in component of degree $k$.
Moreover, after fixing a direction $h_0$ in the Cartan subalgebra we fixed a decomposition
$$\La = \La^{+}\oplus \La^{0}\oplus \La^{-}\ :\  \langle h_0,\lambda_+\rangle > 0 > \langle h_0,\lambda_{-}\rangle 
\text{ for } \lambda_{+}\in wt(\La^{+}), \lambda_-\in wt(\La^{-}), \ [\h,\La^{0}]=0 $$
into the sum of subalgebras of positive, zero and negative $\h$-weight and, using that we defined a partial ordering
$\leq_{\La}$ on the set of integral dominant weights (see~\eqref{eq::def::ordering}).
\begin{definition}
 We say that a graded module $M\in \famod$ is a \emph{highest weight representation} of graded weight $(\lambda,k)$
 if there exists a cyclic vector $v_{\lambda}\in M_{k}^{\lambda}$ 
such that $M_{l}^{\mu}\neq 0$ implies that $\mu \leq_{\La} \lambda$.

We say that $v\in M_{k}$ is a \emph{highest weight vector} of
graded weight $(\lambda,k)$
if it is a vector of $\h$-weight $\lambda$  and all weights in the submodule 
$U(\La)\cdot v\subset M$ generated by $v$ are less or equal than $\lambda$ with respect to the partial ordering $\leq_{\La}$.
Moreover, we say that the highest weight weight vector $v$ is a \emph{pure} 
if the dimension of the weight subspace $\lambda$ in submodule $(U(\La)\cdot v)$ is equal to one 
and, therefore, $(U(\La)\cdot v)^{\lambda}$ is spanned by $v$.
\end{definition}

\begin{remark}
 For the usual BGG category $\calO$ the highest weight vector always spans the corresponding highest weight subspace
and, therefore, one does not need to separate the notions of highest weight vector and a pure highest weight vector in that case.
\end{remark}

\begin{lemma}
For all modules $M\in\famod$
a vector $v\in M$ of weight $\lambda$ is 
 a highest weight vector if and only if
the action of $U(\n^{+}\oplus {\La_{>0}^{+}})$ is trivial.
Moreover, it is a pure highest weight vector if and only if the action of $U(\La_{>0}^{0})$ is also trivial.
\end{lemma}
\begin{proof}
The subalgebras $\n^{+}$ and $\La^{+}_{>0}$ are generated by elements of nonzero nonnegative weights with respect to the partial ordering we choose. 
Therefore, for each generator $g\in \n^{+}\oplus {\La^{+}}$ the weight of the vector $g\cdot v$ is 
greater or equal than the weight of $v$. 
Respectively, the subspace  $\n^{-}\oplus\La_{>0}^{-}\oplus\La_{>0}^{0}$ 
contains only zero and negative weights and, therefore, generates the highest weight module.
We will give the description of a universal highest weight module in the next Section~\ref{sec::stand::mod}
\end{proof}

\subsection{Standard, proper standard and costandard modules}
\label{sec::stand::mod}
Let us define a \emph{standard} module $\Delta(\lambda)$ as a module generated by a cyclic vector $p_{\lambda}$ of degree $k$
modulo following relations:
$$
\Delta(\lambda,k):= U(\La) p_\lambda
\left/
\left(
\begin{array}{c}
\n^{+} p_\lambda = \La_{>0}^{+} p_\lambda = 0, \\
 \forall h\in \h, h p_\lambda=\lambda(h) p_{\lambda}, \\
 f_i^{\lambda(h_i)+1} p_\lambda = 0, \ \ \forall i\in\{1,\ldots, r=rk(\La_{0})\}
\end{array}
\right)
\right.
$$
The quotient module $\Delta(\lambda,k)/(U(\La_{>0}^{0}) p_\lambda=0)$ is called a \emph{proper standard} module
and is denoted by ${\oDelta(\lambda,k)}$.
\begin{lemma}
\label{lem::stand::amod}
 The standard module $\Delta(\lambda,k)$ is the universal highest weight module in $\famod$ of weight $\lambda$.
What means that for all modules $M\in\famod$ and all highest weight vectors $v\in M_{k}$ of weight $\lambda$
the map $p_{\lambda} \mapsto v$ lifts to a map of graded $\La$-modules $\Delta(\lambda,k)\rightarrow M$.

Respectively, the proper standard module ${\oDelta(\lambda,k)}$ is the universal module 
with a pure highest weight vector of weight $\lambda$ and degree $k$.
\end{lemma}
\begin{proof}
The relations $\{f_i^{\lambda(h_i)+1} v=0, i=1\ldots rk(\La_0)\}$ follows from the condition that the $\La_{0}$-submodule generated by $v$
is of finite dimension and, hence, coincides with irreducible $\La_{0}$-module of highest weight $\lambda$.
These are the relations for the induced module $Ind_{\La_0}^{\La}L(\lambda)$ denoted earlier by $\Proj(\lambda)$.
Hence, the standard module $\Delta(\lambda)$ is the quotient module of a module from $\famod$ and, in particular, belongs to $\famod$.
If $v\in M_k$ is a highest weight vector then the relations $\n^{+} v=  \La_{>0}^{+} v = 0$ are satisfied by definition.
The relation $h v = \lambda(h) v$ is the condition on the weight of $v$.
The remaining condition is always satisfied because the action of $\La_0$ is locally finite on modules from $\famod$.
Consequently, the map $p_\lambda \mapsto v$ lifts to a map of graded $\La$-modules. 

The universality of the proper standard module ${\oDelta(\lambda,k)}$ follows from the observation
that the weight subspace of weight $\lambda$ in $\Delta(\lambda,k)$ is generated by the action of $U(\La_{>0}^{0})$ on $p_\lambda$.
\end{proof}

\begin{remark}
Standard modules are simply defined in terms of generators and relations however their graded $\La_{0}$-character may be very complicated.
In particular, we will show in Section~\ref{sec::ex::cur} that for current Lie algebras $\g\otimes \C[t]$ 
the corresponding characters are given by $q$-Hermite polynomials (specialization of Macdonald polynomials for $t=0$).
\end{remark}

Let us describe modules yielding dual universal properties of having a (pure) highest weight vector.
Unfortunately, the dual to the standard module does not belong to $\famod$,
so we are able to say something only about duals of proper standards.

We set a \emph{proper costandard} module ${\onabla(\lambda,k)}$ to be the dual module
$Hom(\oDelta(-\omega_0(\lambda),-k),\C)$ where $\omega_0$ is the longest element of the Weyl group of $\La$.
This module is cogenerated by a cocyclic  lowest weight vector of weight $w_0(\lambda)$ of degree $k$.

\begin{lemma}
 The proper costandard module ${\onabla(\lambda,k)}$ is the couniversal
pure highest weight module in $\famod$ of weight $(\lambda)$.
That is for each module $M\in\famod$ with a pure highest weight vector $v\in M_{k}^{\lambda}$ 
the map $v\mapsto p_{\lambda}$ lifts to a map of graded $\La$-modules $M\rightarrow {\onabla(\lambda,k)}$.
\end{lemma}
\begin{proof}
Let us show that the module ${\oDelta(\lambda,k)}$ is finite-dimensional.
Recall that we assume that the Lie algebra $\La$ is finitely generated. 
Let $d$ be the maximum of the degree of Lie algebra generators of $\La$ and $m$ be the sum of coefficients of 
the decomposition of $\lambda$ into the sum of simple positive weights.
Then, highest weights of the graded components ${\oDelta(\lambda)}_{l}$ of degree $l$ greater than $md$ may be only negative
and, hence,  ${\oDelta(\lambda,k)}_{l}=0$ for $l>md+k$. Thus ${\oDelta(\lambda,k)}$ is finite-dimensional
and, therefore, its dual ${\onabla(-\omega_0(\lambda),k)}:=Hom({\oDelta(\lambda,-k)},\C)$ is also finite-dimensional and belongs to $\famod$.

Let $v$ be a pure highest weight vector of weight $\lambda$ and degree $k$ in a module $M\in\famod$.
In particular $v\in M_k^{\lambda}$ and the module $M_{\leq k}:=M/\left(\ooplus_{s>k} M_s\right)$ is finitedimenaional,
consequently its graded linear dual is finite-dimensional and belongs to $\famod$.
From universality property of proper standard modules we know the existence of a map of graded finite-dimensional $\La$-modules
 $\oDelta(-\omega_0(\lambda),-k) \rightarrow Hom_{\C}(M_{\leq k},\C)$ whose linear dual extends the map $v\rightarrow p_{\lambda}$.
\end{proof}

In the case when the Lie algebra $\La$ has an anti-involution $\tau$ the weight decomposition of $\La$ simplifies.
Indeed, the restriction of $\tau$ on the semisimple part $\La_{0}$ coincides with the Cartan anti-involution. 
Therefore, anti-involution $\tau$ interchanges positive and negative weight spaces and, as assumed, preserves the grading:
$$
\tau(\La_{i}^{\lambda}) = \La_{i}^{-\lambda}.
$$
Consequently, each $\La_{0}$-module $\La_{i}$
is a finite-dimensional selfdual module.
\begin{proposition}
\label{prop::inv::char}
 If the Lie algebra $\La$ admits an anti-involution $\tau$ then $\onabla(\lambda,k)=\oDelta(\lambda,-k)^{\dual}$ 
and, in particular, 
$$\chi_{q}(\onabla(\lambda,k)) = q^{k}\chi_{q}(\onabla(\lambda,0))=q^{k}\chi_{q^{-1}}(\oDelta(\lambda,0))= \chi_{q^{-1}}(\oDelta(\lambda,-k)).$$
\end{proposition}
\begin{proof}
 Direct comparison of highest weight vectors and characters. 
\end{proof}

Let us state several simple properties of standard and costandard modules repeating the ones known for Verma and dual Verma modules 
in the BGG category $\mathcal{O}$:
\begin{proposition}
\begin{itemize}
 \item 
Whenever $\lambda$ is a maximal weight of a module $M\in\famod$
(that is $M^{\lambda}\neq 0$ and $M^{\mu}=0$ for all $\mu>_{\La}\lambda$)
there exist a pair of integers $k,l$ such that $\dim Hom(\Delta(\lambda,k),M)>0$ and $\dim Hom(M,{\onabla(\lambda,l)})>0$.
\item $\dim Hom(\Delta(\lambda,k),{\onabla(\mu,l)})=\delta_{\lambda,\mu}\delta_{k,l}$. 
The unique nontrivial homomorphism from $\Delta(\lambda,0)$ to ${\onabla(\lambda,0)}$ 
factors through the map to (from) irreducible module $L(\lambda)$.
Where $L(\lambda)$ is considered as a unique quotient in degree $0$ of $\Delta(\lambda,0)$ 
and a unique simple submodule in degree zero of ${\onabla(\lambda,0)}$.
\end{itemize}
\end{proposition}
\begin{proof}
Let $k$ be the minimum of degrees such that $M_{k}^{\lambda}\neq 0$. Such a $k$ exists because modules in $\famod$ are of bounded negative degree. 
Let $v$ be a vector in $M_{k}^{\lambda}$. 
Our partial ordering on the set of weights is chosen in such a way
that for any $\h$-eigen vector $g\in\La$ we have that 
the weight of $g\cdot v$ is either greater, lower or equal to the weight of $v$.
The condition on $\lambda$ to be a maximal weight implies that $v$ is a highest weight vector.
Hence, there is a nonzero map $\Delta(\lambda)\rightarrow M$ that lifts the map $p_\lambda\to v$.
Moreover, the vector $v$ becomes a pure highest weight vector in the quotient module $M_{\leq k}:= M/(\oplus_{s>k} M_{s})$. 
Thus, there exists a nonzero composition of morphisms $M \twoheadrightarrow M_{\leq k} \rightarrow {\onabla(\lambda)}$
which maps $v$ to $p_\lambda\in\onabla(\lambda)$.
\end{proof}

With each standard module $\Delta(\lambda)$ we can associate the following left annihilation ideal 
$$
J_{\lambda}:=\{ x\in U(\La_{>0}^{0}): x\cdot p_{\lambda} = 0\}
$$
of the universal subalgebra of zero weight.
Let us denote by $A^\lambda$ the quotient space $U(\La_{>0}^{0})/J^{\lambda}$ and 
by $A^\lambda\ttt mod$ the category of graded left $\La_{>0}$ modules with the trivial action of $J^\lambda$.
By construction, $A^\lambda$ is isomorphic to the weight subspace of $\Delta(\lambda)$ of weight $\lambda$.

\subsection{Goal: Macdonald polynomials as characters of standard modules}
\label{sec::goal::Mac::Weyl}
Recall that in Section~\ref{sec::pairing} with a $\Z_{+}$-graded Lie algebra with anti-involution we associated 
the notion of generalized Macdonald polynomials
using the homological pairing~\eqref{eq::ext::paring} on the ring of graded characters.
The reasonable question may be if corresponding Macdonald polynomials can be realized as characters of a $\La$-module.
The following theorem explains the motivation of introducing standard and proper standard modules.
\begin{theorem}
\label{thm::Mac::Delta}
If the category $\famod(\La)$ is stratified (see Definition~\ref{def::stratified} on page~\pageref{def::stratified})
then standard and proper standard modules form a dual basis with respect to the Ext-pairing~\eqref{eq::ext::paring}
on the Grothendieck group:
\begin{equation}
\label{eq::ext::Delta:Nabla}
\dim Ext^{i}_{\famod}(\Delta(\lambda,k),{\onabla(\mu,l)})=\delta_{i, 0}\delta_{\lambda,\mu}\delta_{k,l}.
\end{equation}
 Macdonald polynomial $P_\lambda$ coincides with the graded character of the 
corresponding proper standard module $\oDelta(\lambda,0)$.
Moreover, the scalar product of $P_{\lambda}$ with itself is the inverse of the $q$-character of the weight subspace of weight $\lambda$
in the standard module $\Delta(\lambda,0)$:
\begin{equation}
\label{eq::const::P:P}
\langle P_\lambda, P_\lambda\rangle_{\La} = \chi_{q}(\Delta(\lambda,0)^{\lambda})^{-1}.
\end{equation}
\end{theorem}
\begin{proof}
The definition  of a stratified category is given in Section~\ref{sec::strat}
 and the Ext-vanishing property~\eqref{eq::ext::Delta:Nabla} is one of the equivalent conditions on a category to be 
stratified as explained in Theorem~\ref{thm::def::strat}.

Let us explain how does vanishing property~\eqref{eq::ext::Delta:Nabla} imply the so called constant term identity~\eqref{eq::const::P:P}.
Indeed, let $f_\lambda$ be the graded $\La_{0}$-character of ${\oDelta(\lambda,0)}$
and $z_{\lambda}(q)$ be the generating series of dimensions of the weight subspace of weight $\lambda$ in the standard module $\Delta(\lambda,0)$.
Note that $z_{\lambda}(q)$ is a power series in $q$ with $z_{\lambda}(0)=1$.
Corollary~\ref{cor::famod::strat} explains the existence of factorization of the characters of the standard module
as a product of its subspace of highest weight and the character of the corresponding proper standard module. 
\begin{equation}
\label{eq::stand::decomp}
\chi_{q\ttt\La_{0}}(\Delta(\lambda,0)) = \chi_{q\ttt\La_{0}}({\oDelta(\lambda,0)}) \cdot \chi_{q}(\Delta(\lambda,0)^{\lambda})
\end{equation}
Consequently, we have the following orthogonality of characters of proper standard modules:
\begin{multline*}
\langle f_{\lambda},f_{\mu}\rangle_{\La} = \langle {\oDelta(\lambda,0)},{\oDelta(\mu,0)}\rangle_{\La} =
\frac{1}{z_{\lambda}(q)} \langle \Delta(\lambda,0)^{\lambda}\otimes {\oDelta(\lambda,0)},{\oDelta(\mu,0)}\rangle_{\La} =
\frac{1}{z_{\lambda}(q)} \langle \Delta(\lambda,0),{\oDelta(\mu,0)}\rangle_{\La} = 
\\
=
\frac{1}{z_{\lambda}(q)}\left(\sum_{i,k} (-1)^{i} q^{k}\dim Ext^{i}_{\famod}(\Delta(\lambda,-k),\onabla(\mu,0))\right) =
\frac{1}{z_{\lambda}(q)}\sum_{k} q^{k}\dim Hom_{\famod}(\Delta(\lambda,-k),{\onabla(\mu,0)}) =
\frac{\delta_{\lambda,\mu}}{z_{\lambda}(q)} 
\end{multline*}
Since $[\oDelta(\lambda,0):L(\lambda)]=1$ and $[\oDelta(\lambda,0):L(\mu)]=0$ for all $\mu\not\leq_{\La}\lambda$
polynomials $f_\lambda$ satisfy both properties of Definition~\ref{def::poly::Macdonald}
and thus $f_\lambda$ is a generalized Macdonald polynomial $P_{\lambda}$.
\end{proof}

\section{Highest weight categories: categorical approach}
\label{sec::bgg::hwc}
The definition of a stratified category we suggest below seems to be a bit different from the standard one given in~\cite{CPS}.
However, the philosophy is the same and all natural examples form a stratified category in both sense.
The basic example which we have in mind is the category $\mathcal{O}$ for modules over semisimple Lie algebra 
introduced by Bernstein-Gelfand-Gelfand (\cite{BGG},\cite{Humpreys}).
Our aim is to use this approach for the category $\famod(\La)$ of graded modules over $\Z_{+}$-graded Lie (super)algebra $\La$ 
and show the criterion of the category $\famod(\La)$ to be stratified therefore we adapt definitions to this particular case.

Moreover, we want to separate two notions: highest weight categories and stratified categories.
By a highest weight category we roughly mean a category of modules with appropriate natural ordering of simple objects.
By a stratified category we assume that in addition to be a highest weight category the BGG reciprocity holds.

\subsection{Category of $\La$-modules as a Highest Weight Category}
\label{sec::famod::hwc}
The category $\famod(\La)$ is our main example of a highest weight category, therefore, we first explain what kind of 
general properties do we need in order to make general conclusions 
and afterwords state the definition of a \emph{highest weight category}.

Recall, that our constructions depends on the following data:
a $\Z_{+}$-graded Lie algebra $\La=\oplus_{k\geq 0}\La_k$ with semisimple zero graded component $\La_0$.
We denoted by $\famod=\famod(\La)$ the category of finitely generated graded modules with locally finite action of $\La_{0}$.
In particular, we considered modules with bounded negative degrees and finite-dimensional graded components.
Irreducible objects in $\famod$ are irreducible finite-dimensional $\La_0$-modules with additional integer grading.
They are indexed by a pair: a dominant weight $\lambda\in\PPdom$ and an integer number $k$.
Starting with a direction $h_0\in\h$ we fixed a natural partial ordering $\leq_{\La}$ on the set of dominant weights~\eqref{eq::def::ordering}.
Let $\famod^{\leq\lambda}$ (respectively $\famod^{<\lambda}$) be the subcategory of modules whose 
irreducibles in the Jordan-H\"older decomposition include only 
$L(\mu,k)$ with $\mu\leq_{\La} \lambda$ (respectively $\mu<_{\La}\lambda$).
The inclusion functor $\imath_{\lambda}:\famod^{\leq\lambda}\rightarrow \famod$ is obviously fully faithful exact functor.
It admits both left and right adjoint functors, however, we need only left adjoint for our purposes:
\begin{lemma}
 The assignment $M \mapsto M/(M^{\not\leq_{\La}\lambda})$\footnote{We factor a module by its submodule generated by (highest) weight vectors of weights 
not less or equal than $\lambda$.} defines a left adjoint functor $\imath^{!}_{\lambda}:\famod \rightarrow \famod^{\leq\lambda}$.
\end{lemma}
\begin{proof}
 If $M\in\famod$ and $N\in\famod^{\leq\lambda}$ then any morphism $\varphi:M\to N$ maps all vectors whose weights 
are not less or equal to $\lambda$ to zero. Therefore, a morphism $\varphi$ factors through the quotient $M/(M^{\not\leq_{\La}\lambda})$.
\end{proof}
In particular, we have $\imath^{!}_{\lambda}\circ \imath_{\lambda} = Id_{\famod^{\leq\lambda}}$ for all $\lambda\in\PPdom$ 
and the collection
of functors $\{\imath_{\lambda}\circ\imath^{!}_{\lambda}\, : \, \lambda\in\PPdom\}$ defines a decreasing filtration on any module $M\in\famod$:
$$
{\cal{F}}^{\lambda} M:= (M^{\not{\leq}\lambda}) = \text{ submodule generated by } M^{\mu} \text{ with } \mu {{\nleq}} \lambda.
$$
We have $\lambda\leq\mu$ implies $ {\cal{F}}^{\lambda} M \supset {\cal{F}}^{\mu} M$ and 
$\imath_{\lambda}(\imath_{\lambda}^{!}(M)) = M /  {\cal{F}}^{\lambda} M$.
Moreover, for all modules $M\in\famod$ we have its presentation as the inverse limit:
$$
M = \varprojlim_{\lambda\in \PPdom} \imath_{\lambda}(\imath_{\lambda}^{!}(M)).
$$
where the limit is taken with respect to the aforementioned partial ordering on $\PPdom$.

Recall that the standard module $\Delta(\lambda,k)$ was defined as the universal highest weight $\La$-module generated 
by a highest weight vector $p_\lambda$
of weight $\lambda$ and degree $k$.
By definition, $\Delta(\lambda,k)$ belongs to $\famod^{\leq\lambda}$ and universality means that it is a projective 
cover of $L(\lambda,k)$ in $\famod^{\leq\lambda}$.
Moreover, the weight subspace of the weight $\lambda$ in the standard module $\Delta(\lambda)$ is 
the quotient of $U(\La_{>0}^{0})$ by annihilation ideal $J^{\lambda}$ and was denoted by $A^{\lambda}$.
\begin{lemma}
\label{lem::La::quotient::cat}
 The quotient category $\famod^{=\lambda}:=\famod^{\leq\lambda}/\famod^{<\lambda}$ is equivalent to the category of graded 
$A^{\lambda}$-modules.
\end{lemma}
\begin{proof}
Let $M$ belongs to $\famod^{\leq\lambda}$.
The universal enveloping algebra $U(\La_{>0}^{0})$ acts on the weight subspace $M^{\lambda}$.
Moreover, this action factors through the annihilation ideal $J^{\lambda}$ which is known to act by zero on 
a universal highest weight module $\Delta(\lambda)$.
Consequently, the assignment $M \mapsto M^{\lambda}$ defines a functor  $\rr_{\lambda}:\famod^{\leq\lambda}\rightarrow A^{\lambda}\ttt mod$.

Notice that the standard module $\Delta(\lambda)=\Delta(\lambda,0)$ is a 
$(\La,\La_{>0}^{0})$-bimodule with respect to the following action:
$$
\forall x\in U(\La),\ y\in U(\La_{>0}^{0}) \ \ u\in U(\La) \ \text{ we set } \ (x,y)\cdot u (p_\lambda):= x(u (y(p_\lambda))
$$
To ensure that this action is well defined it is enough to notice that
for all $y\in U(\La_{>0}^{0})$ the element $y(p_\lambda)$
is a highest weight vector of weight $\lambda$ and, from universality properties of standard modules, we know that
 $u( p_\lambda) =0$ implies $u( y (p_\lambda)) =0$.
Therefore, the assignment $K\mapsto \Delta(\lambda)\otimes_{A^{\lambda}} K$ defines a functor $\ww_{\lambda}:A^{\lambda}\ttt mod \rightarrow \famod^{\leq\lambda}$
that is left adjoint to the functor $\rr_\lambda$:
$$
Hom_{\famod^{\leq\lambda}}(\Delta(\lambda)\otimes_{A^{\lambda}} K, M) = Hom_{A^{\lambda}\ttt mod}(K, M^{\lambda})
$$
Since $\Delta(\lambda)^{\lambda}\cong A_{\lambda}$ we have 
$$\rr_{\lambda}\circ\ww_{\lambda} = (\Delta(\lambda)\otimes_{A^{\lambda}} \ttt)^{\lambda} =  (A^{\lambda}\otimes_{A^{\lambda}} \ttt )^{\lambda} = 
Id_{A_{\lambda}\ttt mod}.
$$
Moreover, $M\in\famod^{<\lambda}$ if and only if $\rr_{\lambda}(M)=M^{\lambda}=0$ and, consequently, the category of graded $A^{\lambda}$-modules 
is canonically isomorphic to the quotient category $\famod^{\leq\lambda}/\famod^{<\lambda}$.
Finally, two modules $M,N\in\famod^{\leq\lambda}$ are isomorphic in the quotient category $\famod^{=\lambda}$
if and only if their $\lambda$-weight subspaces are isomorphic.
\end{proof}
In the proof of above Lemma~\ref{lem::La::quotient::cat} we also defined an exact functor 
$\rr_{\lambda}:\famod^{\leq\lambda}\rightarrow A^{\lambda}\ttt mod$
which is just taking the weight subspace of weight $\lambda$.
Moreover, tensoring from the left by the standard module we defined its left adjoint functor.
We summarize all defined functors in the following diagram:
\begin{equation}
\label{eq::diag::fun::La}
\xymatrix{
\famod 
 \ar@<1.5ex>[rrr]^{M \mapsto M/(M^{\not{\leq}\lambda})}
& &
& \famod^{\leq \lambda}
\ar@<+1.5ex>[lll]_{\bot}^{M \leftarrow M}
\ar@<-1.5ex>[rrr]_{M \mapsto M^{\lambda}}^{\bot}
& &
& 
\famod^{=\lambda} \cong A^{\lambda}\ttt mod
\ar@<-1.5ex>[lll]_{\Delta(\lambda)\otimes_{U(\La_{>0}^{0})} K\leftarrow K}
}
\end{equation}
The subsequent last property explains why we consider a partial ordering on the set of weights rather than the total linear ordering of weights.
\begin{lemma}
\label{lem::hw::La}
The standard module $\Delta(\lambda,k)$ remains to be a projective cover of $L(\lambda,k)$ in the category $\famod^{\leq \{\lambda,\mu_1,\ldots,\mu_k\}}$
for any collection of pairwise incomparable dominant weights $\{\lambda,\mu_1,\ldots,\mu_k\}\subset \PPdom$.
\end{lemma}
\begin{proof}
Let  $M$ belongs to $\famod^{\leq \{\lambda,\mu_1,\ldots,\mu_k\}}$ what means that $M^{\alpha}\neq 0$ implies 
that either $\alpha\leq \lambda$ or $\alpha\leq \mu_1$, or \ldots, or $\alpha\leq \mu_k$.
If the map $M\rightarrow L(\lambda,k)$ of graded $\La$-modules is nontrivial then there exists a vector $v\in M_k$ of weight $\lambda$
whose image $\varphi(v)$ is a highest weight vector of $L(\lambda,k)$.
Since the weights $\lambda$ and $\mu_i$'s are incomparable 
 each vector $v$ of weight $\lambda$ is a highest weight vector and, therefore, there is a map $\Delta(\lambda)\rightarrow M$
which lifts the map $p_\lambda$ to $v$.
Hence, $\Delta(\lambda)$ is projective in $\famod^{\leq \{\lambda,\mu_1,\ldots,\mu_k\}}$
\end{proof}

In the next section we will give the definition of a highest weight category formalizing the existence of Diagram~\eqref{eq::diag::fun::La} 
and property of Lemma~\ref{lem::hw::La}.

\begin{remark}
One can complete a given partial ordering on $\PPdom$ to a total linear ordering of this set.
Then Lemma~\ref{lem::hw::La} may be omitted and canonical filtrations $\cal{F}^{\lambda}$ and $\overline{\cal{F}}_{\lambda}$
 become a decreasing (resp. increasing) $\Z_{+}$-filtrations. 
\end{remark}

\subsection{Abstract categorical setup}
\label{sec::hwcat}
Let $\Cat$ be an abelian category with the set of simples (up to isomorphism) indexed by a set $\Upsilon$ and
assume there is a given map $\rho: \Upsilon \rightarrow \Xi$ to the partially ordered set $(\Xi,\leq)$:
such that for all elements $\lambda\in \Xi$ there are only finitely many $\mu$ that are less or equal than $\lambda$.
The elements of $\Xi$ will be referred as \emph{weights} and the function $\rho$ is called \emph{the weight function}.

For any $\alpha\in\Upsilon$ we denote by $L(\alpha)$ the corresponding simple
and by $\Proj(\alpha)$ a projective cover of $L(\alpha)$ which is unique up to natural isomorphism.
We require the following finiteness assumptions:\footnote{in general there is a big freedom in the possible choice of 
particular finiteness conditions, but some conditions are always have to be required.}
each object $M\in\Cat$ admits a decreasing filtration 
\begin{equation}
\label{eq::Jordan::filtr}
 M=F^{0}M\hookleftarrow F^{1}M  \hookleftarrow F^{2}M \hookleftarrow \ldots \hookleftarrow 
\end{equation}
whose successive quotients $F^{i}M/F^{i+1}M$ are simple objects. 
There is no canonical choice of such a filtration, but what is important for us is that there exists a bounded from below filtration.
The number of times that $L(\alpha)$ appears as a successive quotient is called the Jordan-H\"older multiplicity of $L(\alpha)$ in $M$
and is denoted by $[M:L(\alpha)]$. 
We suppose that this multiplicity is finite for all $\alpha\in\Upsilon$.
In general, this multiplicity may be infinite.
However, the further conclusions about highest weight categories deals with concrete identities for these multiplicities,
so it is better to have a reasonable way to count these multiplicities even if they are infinite.

With any $\lambda\in\Xi$ we associate the Serre subcategory $\Cat^{\leq \lambda}$ ($\Cat^{<\lambda}$) 
generated by irreducibles $L(\mu)$ with $\rho(\mu)\leq \lambda$ (with $\rho(\mu)<\lambda$ respectively).
Similarly, any given subset $S$ of pairwise incomparable elements of $\Xi$
defines the Serre subcategory $\Cat^{\leq S}$ 
generated by irreducibles $L(\mu)$ with $\rho(\mu)$ less or equal than at least one of the elements of 
the subset $S$.
Denote the quotient category $\Cat^{\leq \lambda} / \Cat^{<\lambda}$ by $\Cat^{=\lambda}$.
Let $\imath_{\lambda}: \Cat^{\leq \lambda} \rightarrow \Cat$ be the embedding functor
and $\rr_{\lambda}: \Cat^{\leq \lambda} \rightarrow \Cat^{=\lambda}$ be a projection functor.
This functors are exact by construction.

\begin{definition}
\label{def::hwcat}
An abelian category $\Cat$ with enough projectives is called a \emph{Highest Weight Category} if the following is satisfied:

($h1$) 
for each subset $S\in\Xi$ of pairwise incomparable elements the embedding functor $\imath_{S}:\Cat^{\leq S}\rightarrow \Cat$
admits a left adjoint $\imath^{!}_{S}$ functor such that $\imath^{!}_{S}\circ \imath_{S} = Id_{\Cat^{\leq S}}$
and for each module $M\in\Cat$ the inverse limit $(\varprojlim_{\lambda\in \Xi} \imath_{\lambda}(\imath_{\lambda}^{!}(M)))$ 
with respect to the partial ordering on $\Xi$
 is isomorphic to $M$.

($h2$) 
the projection functor
$\rr_{\lambda}:\Cat^{\leq\lambda}\rightarrow \Cat^{=\lambda}$ has a left adjoint functor $\ww_{\lambda}$.

($h3$) 
for each subset of pairwise incomparable elements $S\subset \Xi$, each $\lambda\in S$ and $\alpha\in\rho^{-1}(\lambda)$
the projective cover of $L(\alpha)$ in $\Cat^{\leq\lambda}$ and in $\Cat^{\leq S}$ are isomorphic in $\Cat$.
The latter object is called standard and is denoted by $\Delta(\alpha)$.
That is 
$$\Delta(\alpha):= \imath_{\lambda}(\imath_{\lambda}^{!}(\Proj(\alpha)) = \imath_{S}(\imath_{S}^{!}(\Proj(\alpha)).$$

($h4$) for all $\alpha\in\Upsilon$ there exists a maximal essential extension of $L(\alpha)$ called proper standard and denoted by $\onabla(\alpha)$
such that the simple subquotients of $\onabla(\alpha)/L(\alpha)$ are of the form $L(\beta)$ with $\rho(\beta)<\rho(\alpha)$.
\end{definition}
In particular,  for any pair of incomparable elements $\lambda,\mu\in\Xi$ and  $\forall\alpha\in\rho^{-1}(\lambda)$ 
we have the following diagram of adjunctions:
\begin{equation}
\label{eq::diagram::functors}
\xymatrix{
\Cat
 \ar@{.>}@<1.5ex>[rr]^{\imath_{\lambda,\mu}^{!}}_{\bot} 
&
&
\Cat^{\leq\{\lambda,\mu\}} 
 \ar@{.>}@<1.5ex>[rr]^{\imath_{\lambda}^{!}}_{\bot} 
 \ar@<1.5ex>[ll]^{\imath_{\lambda,\mu}} 
&
& \Cat^{\leq \lambda}
\ar@<1.5ex>[ll]^{\imath_\lambda}
\ar@<-1.5ex>[rr]_{\rr_{\lambda}}^{\bot}
&
& 
\Cat^{=\lambda}
\ar@{.>}@<-1ex>[ll]_{\ww_{\lambda}}
\\
\Proj(\alpha) \ar@{.>}[rr] && \Delta(\alpha) \ar@{.>}[rr] && \Delta(\alpha) &&
}
\end{equation}
If the ordering on $\Xi$ is total ordering then the middle term $\Cat^{\leq\{\lambda,\mu\}}$ is omitted and the condition ($h3$) disappears.

Similarly to the definition of the standard object with each $\alpha\stackrel{\rho}{\mapsto}\lambda\in\Xi$ 
we set the proper standard object $\oDelta(\alpha)$ to be 
the universal object in $\Cat^{\leq\lambda}$ which surjects on the simple object $L(\alpha)$
and the kernel of this surjection does not have $L(\alpha')$ with $\rho(\alpha')=\lambda$ in the Jordan-H\"older decomposition.
By definition, the proper standard object $\oDelta(\lambda)$ coincides with the image of a simple object under the left adjoint functor $\ww_{\lambda}$:
$$
\text{for } \lambda=\rho(\alpha) \text{ we have }\forall M\in\Cat^{\leq\lambda} \ \ 
Hom_{\Cat^{\leq\lambda}}(\oDelta(\alpha),M) \cong Hom_{\Cat^{=\lambda}}(\rr_\lambda(L(\alpha)),\rr_{\lambda}(M)).
$$
Converse, the proper costandard object $\onabla(\alpha)$ is the image of the right adjoint functor to the projection functor:
$$
\text{for } \lambda=\rho(\alpha) \text{ we have }\forall M\in\Cat^{\leq\lambda} \ \ 
Hom_{\Cat^{\leq\lambda}}(M,\onabla(\alpha)) \cong Hom_{\Cat^{=\lambda}}(\rr_{\lambda}(M),\rr_\lambda(L(\alpha))).
$$
\begin{lemma}
 Let $M$ belongs to $\Cat^{\leq\lambda}$ and   $\rr_{\lambda}(M)\neq 0$ then
there exists an element $\alpha\in\rho^{-1}(\lambda)$ 
such that $\dim Hom_{\Cat}(\Delta(\alpha),M)>0$ and $\dim Hom_{\Cat}(M,\onabla(\alpha))>0$.
\end{lemma}
\begin{proof}
The condition $\rr_{\lambda}(M)\neq 0$ means that there exists $\alpha\in\rho^{-1}(\lambda)$ such that $[M:L(\alpha)]>0$ for $\alpha\in\rho^{-1}(\lambda)$ then
Consider a Jordan-H\"older filtration of $M$ as it is in~\eqref{eq::Jordan::filtr}.
Let $k$ be the minimal integer number such that corresponding quotient $F^{k}M/F^{k+1}M$ is isomorphic to $L(\alpha')$ with $\alpha'\in\rho^{-1}(\lambda)$.
 We have the following commutative diagrams:
$$
\xymatrix{
& \Delta(\alpha') 
\ar@{.>}[dl]_{\exists}
\ar@{->>}[d]
&
&&&&
& \onabla(\alpha') & 
\\
F^{k}M 
\ar@{>>}[r] 
\ar@{^{(}->}[d]
&
L(\alpha')
\ar@{->>}[r]
&
0
&&&&
0 \ar[r] 
& 
L(\alpha') \ar[u] \ar[r] 
& 
M/F^{k+1}M \ar@{.>}[ul]_{\exists} 
\\
M & &&&&& & & M \ar@{>>}[u] 
}
$$
Note that the multiplicity of $L(\alpha')$ in $M/F^{k}M$ is equal to $1$ and no other factors 
with the same weight $\lambda$ appears in Jordan--H\"older decomposition of $M/F^{k}M$. 
Hence, the dotted arrows in the Diagram above 
exists due to universality properties
of standard and proper costandard modules.
\end{proof}
\begin{corollary}
\label{cor::map::costandard}
 For all $M\in\Cat^{<\lambda}$ there exists $\mu$ with $\rho(\mu)<\lambda$ such that $\dim Hom(M,\onabla(\mu))>0$.
\end{corollary}

\begin{example}
Let $\g^{A}=\n^{+}\oplus\h\oplus\n^{-}$ be the Kac--Moody Lie algebra associated with the Cartan matrix $A$ and a root system $\Pi$.
The BGG category $\calO=\calO(\g^A)$ is the subcategory of finitely generated $\g_A$-modules such that
\\
{$\phantom{....}(\calO 1)$} the action of $\h$ is semisimple,
\\
{$\phantom{....}(\calO 2)$} the nilpotent subalgebra $\n^{+}$ acts locally finite.

The irreducible objects in $\calO$ are numbered by $\h$-weights and we denote by $L(\lambda)$ the maximal irreducible quotient of 
the Verma module $Ind_{\h\oplus\n^{+}}^{\g^{A}} \C_{\lambda}$ of the highest weight $\lambda\in \h^{*}$.

\begin{proposition*}(\cite{BGG})
\label{ex::BGG::hwc}
 The BGG category $\calO(\g^A)$ is a highest weight category with respect to the following partial ordering on irreducibles:
$$
\lambda\leq \mu \ \stackrel{def}{\Longleftrightarrow} \ \lambda-\mu\in \Z_{+}(\Phi_{+}), 
\text{ and } \exists \omega\in W: \ \mu = \omega\cdot \lambda
$$
The standard and proper standard modules coincide with corresponding Verma module;
Respectively costandard modules are isomorphic to dual Verma modules:
$$
\Delta(\lambda)=\oDelta(\lambda)= Ind_{\h\oplus\n^{+}}^{\g^{A}} \C_{\lambda}; \quad \onabla(\lambda)= \Delta(\lambda)^{\dual} 
$$
\end{proposition*}
\end{example}

Previous Subsection~\ref{sec::famod::hwc} was devoted to explain the following
\begin{proposition}
The partial ordering $\geq_{\La}$ on the set of dominant weights $\PPdom$ defined by~\eqref{eq::def::ordering}
and the forgetful map $\PPdom\times \Z \rightarrow \PPdom$ from the set of irreducibles 
$\{L(\lambda,k): \lambda\in\PPdom, k\in\Z\}$ defines a structure of a highest weight category on $\famod(\La)$.
\end{proposition}

\subsection{Stratified categories}
\label{sec::strat}
In this section we give a definition of a stratified category and then show several equivalent properties 
for a highest weight category to be stratified.
The notion of \emph{stratified categories} was introduced in~\cite{CPS} 
motivated by generalization of BGG reciprocity discovered for the category $\mathcal{O}$ in~\cite{BGG}.

\begin{definition}
\label{def::stratified}
 A highest weight category $\Cat$ is called \emph{stratified} iff for all 
$\alpha\in \Upsilon$ there exists an epimorphism $\Proj(\alpha)\twoheadrightarrow \Delta(\alpha)$ whose kernel 
admits a filtration with successive quotients isomorphic to $\Delta(\beta)$ with $\rho(\beta)>\rho(\alpha)$.
\end{definition}

\begin{example}{\cite{BGG}}
 The  BGG category $\calO(\g)$ (Example~\eqref{ex::BGG::hwc}) is stratified.
\end{example}

The following Theorem gives useful criterions on a highest weight category to be stratified.

\begin{theorem}
\label{thm::def::strat}
 The following conditions on a highest weight category $\Cat$ are equivalent
\begin{enumerate}
 \item[(S1)] The category $\Cat$ is stratified.
 \item[(S2)]
 The functor $\imath_{\lambda}$ between derived categories $\calD(\Cat^{\leq\lambda})\rightarrow\calD(\Cat)$ is fully faithful.
In other words, for all modules $M,N\in\Cat^{\leq\lambda}$ 
the derived homomorphisms in the categories $\Cat^{\leq \lambda}$ and $\Cat$ are the same:
$$
 Ext^{\udot}_{\Cat^{\leq\lambda}}(M,N) = Ext^{\udot}_{\Cat}(\imath_{\lambda}(M),\imath_{\lambda}(N)).
$$ 
 \item[(S3)] For all $\lambda,\mu\in \Upsilon$ we have the following vanishing conditions
\begin{equation}
 \dim Ext_{\Cat}^{i}(\Delta(\lambda),\onabla(\mu)) = \left\{
\begin{array}{l}
 1, \text{ if } \lambda=\mu \ \& \ i=0, \\
0, \text{ otherwise. }
\end{array}
\right.
\end{equation}
\item[(S4)] The second extension group $Ext_{\Cat}^{2}(\Delta(\lambda),\onabla(\mu))=0$ vanishes for all $\lambda,\mu\in\Upsilon$.
\end{enumerate}
\end{theorem}
\begin{proof}
We will prove implications in the following order: 
$(S1)\Rightarrow (S2) \Rightarrow (S3) \Rightarrow (S4) \Rightarrow (S1)$

{$[(S1)\Rightarrow(S2)]$}
The triangulated category $\calD^{-}_{\leq\lambda}(\Cat)$ is the category of bounded from above complexes of modules from $\Cat$
whose homology groups belongs to $\Cat^{\leq\lambda}$. Its left orthogonal triangulated subcategory ${}^{\perp}\calD^{-}_{\leq\lambda}(\Cat)$
is generated by projective modules $\Proj(\nu)$ with $\rho(\nu)\not\leq\lambda$ because
$$\forall \mu \text{ with } \rho(\mu)\not\leq\lambda \ \ RHom_{\Cat}(\Proj(\mu),M)=Hom_{\Cat}(L(\mu),M) = 0  
\ \Longleftrightarrow \ M\in \calD^{-}_{\leq\lambda}(\Cat).$$
Let us show that for each projective module $\Proj(\mu)$ we have an exact triangle
$X\to\Proj(\mu)\to \imath^{!}_{\lambda}(\Proj(\mu))$ with $X\in{}^{\perp}\calD^{-}_{\leq\lambda}(\Cat)$.
Indeed, it is enough to show that the kernel of the map $\Proj(\mu)\to \imath^{!}_{\lambda}(\Proj(\mu))$ admits a projective resolution
with successive quotients isomorphic to $\Proj(\nu)$ with $\nu\not\leq\lambda$.
If $\rho(\mu)$ is not less or equal to $\lambda$ then $\imath^{!}_{\lambda}(\Proj(\mu))=0$ and the exact triangle collapses to 
a trivial short exact sequence 
$\Proj(\mu)\to\Proj(\mu)\to 0$.
If $\rho(\mu)\leq\lambda$ then the condition $(S1)$ of a stratified category implies that 
 the kernel of the map $\Proj(\mu)\to \imath^{!}_{\lambda}(\Proj(\mu))$ admits a filtration with successive quotients
isomorphic to $\Delta(\nu)$ with $\rho(\nu)\not\leq\lambda$.
Hence, the considered kernel has a projective cover $\Proj$ with successive quotients isomorphic to $\Proj(\nu)$ with $\rho(\nu)\not\leq\lambda$.
The kernel of the map $\Proj\to\Proj(\mu)$ also admits a filtration with successive quotients isomorphic to $\Delta(\nu)$ with 
$\rho(\nu)\not\leq\lambda$.
Using the induction by partially ordered set $\Xi$ we get a resolution 
$$\Proj^{\udot}\rightarrow \Proj(\mu) \rightarrow \imath^{!}_{\lambda}(\Proj(\mu)) \rightarrow 0,$$
such that for all $i$ the corresponding module $\Proj^{i}$ admits a filtration with successive quotients isomorphic to $\Proj(\nu)$ with $\rho(\nu)\not\leq\lambda$.

 Consequently, the left adjoint functor $\bar{\imath}_{\lambda}^{!}:\calD^{-}(\Cat)\rightarrow \calD^{-}_{\leq\lambda}(\Cat)$ 
to the embedding functor between triangulated categories has zero higher derived on projectives and, therefore,  
coincides with the left derived of the corresponding adjoint functor between abelian categories 
$L^{\udot}\imath_{\lambda}^{!}: \calD^{-}(\Cat) \rightarrow \calD^{-}(\Cat^{\leq\lambda})$.
In particular, this implies that $\calD^{-}(\Cat^{\leq\lambda})$ and $\calD^{-}_{\leq\lambda}(\Cat)$ are isomorphic.

$[(S2)\Rightarrow(S3)]$
 In order to show the vanishing of derived homomorphisms between standard and proper costandard it is enough to use their universality.
Indeed, if $\rho(\lambda)\geq \rho(\mu)$ 
then $\Delta(\lambda)$ is projective in $\Cat^{\leq\rho(\lambda)}$.
Hence, $Ext^{>0}_{\Cat}(\Delta(\lambda),\onabla(\mu))=Ext^{>0}_{\Cat^{\leq\rho(\lambda)}}(\Delta(\lambda),\onabla(\mu))=0$.
Similarly, if $\rho(\mu)\geq\rho(\lambda)$ then $Ext^{>0}_{\Cat^{\leq\rho(\mu)}}(\Delta(\lambda),\onabla(\mu))=0$ 
because $\onabla(\mu)$ is almost injective in 
${\Cat^{\leq\rho(\mu)}}$. 
If $\rho(\lambda)$ and $\rho(\mu)$ are incomparable then the condition $(h3)$ of a highest weight category says 
that $\Delta(\lambda)$ is projective in $\Cat^{\leq\{\rho(\lambda),\rho(\mu)\}}$ and we also get vanishings of higher extension groups.

$[(S3)\Rightarrow(S4)]$
Condition $(S4)$ is a particular case of $(S3)$ and we just pointed it out in order to underline that it is enough to check 
the Ext-vanishing only in degree $2$.

$[(S4)\Rightarrow(S1)]$
Let $\Proj(\mu)$ be the projective cover of $L(\mu)$ and let $\Proj(\mu)_{\leq \lambda}:=\imath_{\lambda}^{!}\Proj(\mu)$
 be its image in $\Cat^{\leq\lambda}$.
Let us prove by induction in $\lambda$ that $\Proj(\mu)_{\leq\lambda}/ \Proj(\mu)_{<\lambda}$ has a filtration 
with successive quotients isomorphic to $\Delta(\alpha)$ with $\rho(\alpha)=\lambda$.
Indeed, denote by $\Proj(\mu)_{\lambda}$ the maximal quotient of $\Proj(\mu)_{\leq\lambda}$ whose 
irreducibles are isomorphic to $L(\alpha)$ with $\rho(\alpha)=\lambda$.
The projective cover $\Delta(\Proj(\mu)_{\lambda})\twoheadrightarrow \Proj(\mu)_{\lambda}$ admits a filtration with successive quotients isomorphic to 
standard modules $\Delta(\alpha)$ with $\rho(\alpha)=\lambda$.
Since the quotient $\Proj(\mu)_{\leq\lambda} / \Proj(\mu)_{<\lambda}$ is generated by irreducibles whose weight is equal to $\lambda$
 we get the following exact sequence
\begin{equation}
\label{eq::kernel::proj}
\xymatrix{
0 \ar[r] & K \ar[r] &
 \Delta(\Proj(\mu)_{\lambda}) \ar[r]^{\kappa} & \Proj(\mu)_{\leq\lambda} \ar[r] & \Proj(\mu)_{<\lambda} \ar[r]& 0
}
\end{equation}
where $K$ is the kernel of the map $\kappa$.
Note that $K$ belongs to $\Cat^{<\lambda}$. 
Thus, thanks to Corollary~\ref{cor::map::costandard} if $K$ is different from $0$ then there exists $\nu$ with $\rho(\nu)<\lambda$
such that $\dim Hom(K,\onabla(\nu))>0$.
Let us apply $RHom_{\Cat^{\leq \lambda}}(\ttt,\onabla(\nu))$ to the exact sequence~\ref{eq::kernel::proj} and get a contradiction.

Indeed, first, $\forall \alpha\in\rho^{-1}(\lambda)$ 
there are no nontrivial homomorphisms $\Delta(\alpha)\rightarrow \onabla(\nu)$ since $\nu<\lambda$
and, consequently, no nontrivial homomorphisms from $\Delta(\Proj(\mu)_{\lambda})$ to $\onabla(\nu)$.
Second, $Ext^{1}(\Delta(\mu),\onabla(\nu))$ is always zero, because for each short exact sequence
$0\to \onabla(\nu) \to Q \to \Delta(\mu)\to 0$ the middle object $Q$ belongs either to $\Cat^{\leq\rho(\mu)}$ if $\rho(\mu)\geq\rho(\nu)$,
or to $\Cat^{\leq\{\rho(\mu),\rho(\nu)\}}$ if $\mu$ and $\nu$ are incomparable, or belongs to $\Cat^{\leq\rho(\nu)}$
if $\rho(\nu)>\rho(\mu)$. In the first two cases one has the splitting map $\Delta(\mu)\rightarrow Q$ because it is projective in $\Cat^{\leq\rho(\mu)}$ 
and in $\Cat^{\leq\{\rho(\mu),\rho(\nu)\}}$. In the last case $\rho(\nu)>\rho(\mu)$ there exists a splitting map $Q\rightarrow\onabla(\nu)$ 
because the multiplicity of 
$L(\nu)$ in both $\onabla(\nu)$ and in $Q$ is equal to $1$ and $\onabla(\nu)$ is injective hull with this property.
The condition $(S4)$ says that $Ext^{2}(\Delta(\mu),\onabla(\nu))=0$ for all $\mu$.
From the induction hypothesis we know that $\Proj(\mu)_{<\lambda}$ admits a filtration with successive quotient isomorphic to $\Delta(\beta)$ 
with $\rho(\mu)\leq \rho(\beta)<\lambda$. Therefore, we have the following vanishing
$$Ext_{\Cat^{\leq\lambda}}^{2}(\Proj(\mu)_{<\lambda},\onabla(\nu))=Ext_{\Cat^{\leq\lambda}}^{1}(\Proj(\mu)_{<\lambda},\onabla(\nu))=0$$
All higher derived homomorphisms from $\Proj(\mu)_{\leq\lambda}$ to any module in $\Cat^{\leq\lambda}$ vanishes because
it is projective in $\Cat^{\leq\lambda}$.
Therefore, from the long exact sequence of $RHom(\ttt,\onabla(\nu))$ we see that there is no possibility to have a nontrivial homomorphism 
from $K$ to $\onabla(\nu)$ what may happen only if $K=0$.
\end{proof}

\begin{corollary}
 If a highest weight category $\Cat$ is stratified then the functor $\ww_{\lambda}: \Cat^{=\lambda}\rightarrow \Cat^{\leq\lambda}$ 
is exact for all $\lambda$.
\end{corollary}
\begin{proof}
Since the category is stratified for all $\mu$ we have that $\imath_{\lambda}^{!}(\Proj(\mu))$ admits a filtration whose successive quotients 
are isomorphic to $\Delta(\nu)$ with $\rho(\mu)\leq \rho(\nu)\leq \lambda$. Consequently, the quotient category $\Cat^{=\lambda}$ is generated
by $\Delta(\alpha)$ with $\rho(\alpha)=\lambda$. In particular, 
for any $M\in\Cat^{=\lambda}$ its image $\ww_{\lambda}(M)$ admits a resolution $Q^{\udot}$
by standard modules $\Delta(\alpha)$ with $\rho(\alpha)=\lambda$. Since $\rr_{\lambda}$ is an exact functor and maps projectives to projectives.
We get $\rr_{\lambda}(\ww_{\lambda}(Q^{\udot}))$ is a projective resolution of $M = \rr_{\lambda}(\ww_{\lambda}(M))$.
Moreover, whenever $\rho(\alpha)=\lambda$ we have $\ww_{\lambda}(\rr_\lambda(\Delta(\alpha)))=\Delta(\alpha)$ and $\ww_{\lambda}$ maps 
this resolution to a projective resolution $\ww_{\lambda}(\rr_\lambda(Q^{\udot}))=Q^{\udot}$. 
Therefore, all higher derived functors of $\ww_{\lambda}$ vanishes 
and, thus, $\ww_\lambda$ is exact.
\end{proof}

\subsection{Criterion for $\famod(\La)$ to be stratified}
\label{sec::famod::strat}
Let us present several necessary and sufficient conditions on a graded Lie algebra $\La$ 
in order to have the BGG reciprocity for the highest weight category $\famod(\La)$.
We will illustrate how to use these conditions on the number of examples in the subsequent Section~\ref{sec::examples} below.

Recall that for a graded Lie algebra $\La$ with anti-involution we defined the generalization of Macdonald pairing~\eqref{eq::ext::paring}
and defined the generalized Macdonald polynomial~\ref{def::poly::Macdonald} as orthogonalization of the Schur basis $s_\lambda=\chi_{\La_0}(L(\lambda))$
with respect to the partial ordering on weights~\eqref{eq::def::ordering}.

\begin{theorem}
\label{thm::MacWeyl::strat}
The category $\famod(\La)$ for a graded Lie algebra with anti-involution is stratified if and only if 
for all $\lambda\in\PPdom$ the character of the proper standard module $\onabla(\lambda)$ 
is given by the generalized Macdonald polynomial $P_{\lambda}$ and 
the character of standard module $\Delta(\lambda)$ is given by the generalized dual Macdonald polynomial $Q_{\lambda}$.
\end{theorem}
\begin{proof}
The only if part of this theorem was explained in Section~\ref{sec::goal::Mac::Weyl}. Let us explain the implication in the opposite direction.
Recall, that projective covers of irreducibles $\{\Proj(\lambda):\lambda\in\PPdom\}$ and irreducibles $\{L(\lambda):\lambda\in\PPdom\}$
form a pair of dual basis in the Grothendieck ring.
Respectively their characters form a pair of dual basis in the ring of symmetric functions.
 Macdonald polynomials $P_{\lambda}$ and dual Macdonald polynomials form another pair of dual basis in the ring of symmetric functions.
Let $m_{\lambda,\mu}$ be the transition matrix from Macdonald polynomials to Schur functions. That is:
$$
P_{\mu} = \sum_{\lambda} m_{\lambda,\mu}(q) s_{\lambda} = \sum_{\lambda}\sum_{k\geq 0} m_{\lambda,\mu,k} q^{k}s_{\lambda}
\Longrightarrow
\langle \chi(\Proj(\lambda)),P_{\mu}\rangle = m_{\lambda,\mu}(q).
$$
Consequently,
\begin{equation}
\label{eq::trans::matr}
\chi(\Proj(\lambda)) = \sum_{\mu} m_{\lambda,\mu} \frac{P_{\mu}}{\langle P_{\mu},P_{\mu}\rangle} = 
\sum_{\mu} m_{\lambda,\mu} Q_{\mu} = \sum_{\mu}\sum_{k\geq 0} m_{\lambda,\mu,k} q^{k} Q_{\mu}. 
\end{equation}
Consider the standard filtration  $\imath_{\mu}(\imath_{\mu}^{!}(\Proj(\lambda)))$ of a projective cover $\Proj(\lambda)$.
We have that $\Proj(\lambda)$ has a filtration with successive quotients isomorphic to the factor of $\Delta(\mu,k)^{c_{\lambda,\mu,k}}$
by the module denoted by $K_{\lambda,\mu,k}$.
Here 
$$c_{\lambda,\mu,k} = \dim Hom_{\La}(\Proj(\lambda),\onabla(\mu,-k)) = \dim Hom_{\La_0}(L(\lambda),\onabla(\mu,-k)) = 
\dim Hom_{\La_0}(\oDelta(\mu,k),L(\lambda))
$$
Thus, we have equality:
\begin{multline*}
\sum_{\mu}\sum_{k\geq 0} m_{\lambda,\mu,k} q^{k} Q_{\mu} = \chi(\Proj(\lambda)) =
\chi(\oplus\Delta(\mu,k)^{m_{\lambda,\mu,k}}/K_{\lambda,\mu,k}) =
\sum_{\mu}\sum_{k\geq 0} m_{\lambda,\mu,k} q^{k} Q_{\mu} - \chi(K_{\lambda,\mu,k}). 
\end{multline*}
Since $\chi(K_{\lambda,\mu,k})$ are Schur positive we conclude that $K_{\lambda,\mu,k}=0$.
Consequently, the projective cover $\Proj(\lambda)$ admits a filtration with successive quotients isomorphic to standard modules
and the triangular condition follows from the observation that the multiplicity of an irreducible $L(\lambda,k)$ in the 
proper standard module $\oDelta(\mu,0)$ is different from zero only if $\lambda\leq\mu$ because $\oDelta$ is a highest weight module.
Finally, we have that the category $\famod$ is stratified.
\end{proof}

\begin{corollary}
\label{cor::ADE::type}
 The highest weight category of the modules over Lie algebra $\g\otimes \C[x]$ with $\g$ simply laced is stratified.
For non-simply laced $\g$ the result follows from~\cite{CI}.
\end{corollary}
\begin{proof}
G.\,Fourier and P.\,Littlemann proved in~\cite{FL} that for a simply laced $\g$ 
the proper standard modules (also known as local Weyl modules) coincide with Demazure modules
and their characters are given by Macdonald polynomials.
For non-simply laced $\g$ the result about characters and even about BGG reciprocity is announced in~\cite{CI}.
\end{proof}

The proofs of BGG reciprocity for $\sl_n\otimes \C[x]$ given in~\cite{BBCKL} and in~\cite{CI} for general $\g$
are similar to the methods used in our proof of Theorem~\ref{thm::MacWeyl::strat} and Corollary~\ref{cor::ADE::type}
However, it is known that for $\g$ not simply laced  local Weyl modules are greater than Demazure modules.
We insist that it is better to find another criterion of BGG reciprocity in order to prove it for $\g\otimes\C[x]$ using general categorical setup.

\begin{proposition}
\label{cor::famod::strat}
\begin{enumerate}
 \item 
 If the highest weight category $\famod(\La)$ is stratified then the following properties are satisfied:

$\phantom{.....}$ $(f1)$
 the relative Lie algebra cohomology of $\La$ is equal to the absolute cohomology of the Lie algebra 
$\La^{0}/([\La^{+},\La^{-}])$:
\begin{equation}
\label{eq::star::van::H}
H^{\udot}(\La,\La_0;\C) = H^{\udot}(\La^{0}/(\La^{0}\cap [\La^{+},\La^{-}]);\C)
\end{equation} 

$\phantom{.....}$ $(f2)$
The graded characters of standard and proper standard modules yield the factorization property 
for all $\lambda\in\PPdom$:
\begin{equation}
 \label{eq::factor::Weyl}
\chi_{q\ttt\La_0}(\Delta(\lambda,0)) = \chi_{q\ttt\La_0}({\oDelta(\lambda,0)}) \cdot \chi_{q} (\Delta(\lambda,0)^{\lambda}).
\end{equation}
Moreover, if the Lie algebra $\La$ admits an anti-involution and the notion of generalized Macdonald polynomials $P_{\lambda}$ 
are well defined and 

$\phantom{.....}$ $(f3)$ the generating series of the weight subspace $\chi_{q} (\Delta(\lambda,0)^{\lambda})$ is equal to 
the inverse of the constant term $\langle P_{\lambda},P_{\lambda}\rangle_{\La}$.

\end{enumerate}
\end{proposition}
Before going to the proof of Proposition we state 
\begin{conjecture}
\label{conj::gen::strat}
Properties $(f1)$--$(f3)$ implies that the highest weight category $\famod(\La)$ is stratified
(for a graded Lie algebra with anti-involution).
\end{conjecture}

\begin{proof}
We notice that $0$ is the unique minimal element in $\PPdom$. 
The categories $\Cat^{\leq 0}$ and $\Cat^{=0}$ are the same and coincide with the category of graded $\La$-modules
with the trivial action of $\La^{+}$ and $\La^{-}$. 
Consequently, the category $\Cat^{=0}$ is the category of modules over the quotient of the Lie algebra $\La$ by the ideal generated by subalgebras
$\La^{+}$ and $\La^{-}$. The latter Lie algebra coincides with the quotient of the Lie algebra $\La^{0}$ by the ideal generated by 
$[\La^{+},\La^{-}]$.
Thus the coincidence of group of derived endomorphisms of $\C$ in $\famod^{\leq 0}$ and in $\famod$ looks as follows:
$$
H^{\udot}(\La,\La_{0};\C)= Ext^{\udot}_{\famod(\La)}(\C,\C) = Ext^{\udot}_{\frac{\La^{0}}{(\La^{0}\cap [\La^{+},\La^{-}])}\ttt mod}(\C,\C) 
=H^{\udot}(\frac{\La^{0}}{(\La^{0}\cap [\La^{+},\La^{-}])};\C).
$$
and we get the condition $(f1)$
In particular, we have $\oDelta(0,0)$ is a trivial module $\C$ and $\Delta(0,0)$ is isomorphic to the universal enveloping 
$U(\frac{\La^{0}}{(\La^{0}\cap [\La^{+},\La^{-}])})$.
In particular, if the Lie algebra $\La^{0}/(\La^{0}\cap[\La^{+},\La^{-}])$  is trivial (as it is in all our examples) 
then the relative cohomology of $\La$ are necessary trivial.

Recall from Lemma~\ref{lem::La::quotient::cat} 
that the category $\famod^{=\lambda}$ is canonically isomorphic to the category of $A^{\lambda}$-modules.
Where $A^{\lambda}$ is the quotient of the universal enveloping algebra
$U(\La^{0}_{>0})$ by the annihilation ideal $J^{\lambda}$ and is canonically isomorphic 
to the weight subspace $\Delta(\lambda,0)^{\lambda}$ of weight $\lambda$. 
The PBW theorem implies that any element of the module $\Delta(\lambda,k)$ may be presented in the form $x\cdot y\cdot p_\lambda$
with $y\in U(\La^{0}_{>0})$ and $x\in U(\La^{-})$
Consequently, there is a surjection of vector spaces $\oDelta(\lambda,k)\otimes A^{\lambda} \twoheadrightarrow \Delta(\lambda,k)$.
Note that irreducibles in the category of graded $A^{\lambda}$ modules 
are just trivial one-dimensional modules placed in appropriate degree.
Hence, the graded $A^{\lambda}$-module
is projective if and only if it is a free $A^{\lambda}$ module.
Thus, the functor 
$$\ww_{\lambda} = \Delta(\lambda)\ootimes_{U(\La^{0}_{>0})}\ttt: A^{\lambda}\ttt mod \rightarrow \famod^{\leq\lambda}$$
is exact if and only if we have factorization of characters~\eqref{eq::factor::Weyl}.
This leads to the condition $(f2)$.

The third condition $(f3)$ has been already proven in Theorem~\ref{thm::Mac::Delta}.
\end{proof}

\subsection{Excellent filtrations, BGG reciprocity and Tilting modules}
\label{sec::BGG::appl}
\label{sec::filtrations}

In this subsection we recall some classical properties of stratified categories.
All statements as well as their proofs seems to be standard and may be found in the literature
(see e.g.~\cite{CPS,Humpreys,Losev_Webster}). 
However, to make the text self contained we sketch the proofs as well.
In this section we assume that category $\Cat$ is a highest weight category.

Recall that any module $M\in\Cat$ has a filtration by partially ordered set $\Xi$:
\begin{equation}
\label{eq::filtr::stand}
\lambda\geq\mu \Rightarrow \imath_{\lambda}{\imath_{\lambda}^{!}} M\twoheadrightarrow 
\imath_{\mu}{\imath_{\mu}^{!}} M \text{ and }
 M = \varprojlim_{\lambda\in\Xi} \imath_{\lambda}{\imath_{\lambda}^{!}} M.
\end{equation}
Each successive quotient in this filtration is generated by irreducibles objects with the same weight.
Hence, there exists a subfiltration such that successive quotients are quotients of standard modules:
\begin{equation}
\label{eq::filtr::st::kernel}
\forall M\in\Cat \ \ \exists \calF^{\udot} M  \text{ such that }  
gr \calF^{\udot}M \cong \oplus_{\alpha\in\Upsilon} \Delta(\alpha)^{m_\alpha}/K_{\alpha}  
\end{equation}
where the dimension $m_\alpha$ may be computed using proper costandard modules.
For $\lambda=\rho(\alpha)$ we have:
\begin{multline}
\label{eq::filtr::dim::factors}
m_{\alpha} =\dim Hom_{\Cat^{=\lambda}}\left(\rr_{\lambda}(\imath_{\lambda}^{!}(M)),\rr_{\lambda}L(\alpha)\right) 
= \dim Hom_{\Cat^{\leq\lambda}}\left(\imath_{\lambda}^{!}(M),\onabla(\alpha)\right) =
\\ = \dim Hom_{\Cat}(M,\imath_{\lambda}(\onabla(\alpha))) =\dim Hom_{\Cat}(M,\onabla(\alpha))
\end{multline}
The first equality says that $m_{\alpha}$ is the number of subquotients in $\imath_{\lambda}^{!}M$ generated by $L(\alpha)$.
The second equality follows from the universal properties of proper standard modules 
We use the adjunction of $\imath_{\lambda}^{!}$ and $\imath_{\lambda}$ in the third equality. 
It is natural to discuss modules with empty kernels $K_\alpha$ for all $\alpha\in\Upsilon$ as a separate subcategory.
\begin{definition}
We say that a module $M\in\Cat$ has an \emph{excellent filtration} (or \emph{$\Delta$-filtration})
 if there exists a filtration whose successive quotients are standard modules.
\end{definition}
Note that a module $M$ admits an excellent $\Delta$--filtration if and only if for all $\lambda\in\Xi$ the quotient 
$\imath_{\leq\lambda}{\imath_{\leq\lambda}^{!}} M/(\imath_{<\lambda}{\imath_{<\lambda}^{!}} M)$ 
has a filtration with successive quotients isomorphic to $\Delta(\alpha)$ with $\rho(\alpha)=\lambda$.

\begin{proposition}
\label{prop::exc::filtr}
If the category $\Cat$ is stratified and
a module $M\in\Cat$ admits an excellent $\Delta$-filtration if and only if $Ext^{1}_{\Cat}(M,\onabla(\alpha))=0$ for all $\alpha\in\Upsilon$.
In this case the number $[M:\Delta(\alpha)]$ of successive quotients of $M$ isomorphic to $\Delta(\alpha)$ 
is equal to the dimension of $Hom_{\Cat}(M,\onabla(\alpha))$.
\end{proposition}
\begin{proof}
The category $\Cat$ is stratified and, therefore, we 
can use the vanishing of higher extension groups between standard and proper standard modules
that is an equivalent condition ($s4$) of Theorem~\ref{thm::def::strat}.
Therefore, for each short exact sequence $0\to M\to N \to \Delta(\lambda)\to 0$ we have
$Ext^{1}(N,\onabla(\mu)) = Ext^{1}(M,\onabla(\mu))$ for all $\mu$.
Moreover, if both $M$ and $N$ admits an excellent filtration
we have identity for multiplicities 
$$[M:\Delta(\lambda)] = [N:\Delta(\lambda)]+1, \quad \text{ and } \quad \dim Hom(M,\onabla(\lambda)) = \dim Hom(N,\onabla(\lambda))+1. $$
Note that the number of successive quotient isomorphic to the standard module $\Delta(\lambda)$ with given $\lambda$ is finite.
Therefore, the only if part is proved by induction by the number of successive quotients.

Assume that a module $M$ has zero $Ext^{1}(M,\onabla(\mu))$ for all $\mu$. 
 By induction we may assume that $M$ belongs to $\Cat^{\leq\lambda}$ and 
that for objects of $\Cat^{<\lambda}$ the statements of the proposition are proven.
Consider a short exact sequence
$$
\xymatrix{
0 \ar[r] & \imath_{<\lambda}(\imath_{<\lambda}^{!}(M)) \rightarrow M \ar[r] & M/{\imath_{<\lambda}(\imath_{<\lambda}^{!}(M))} \ar[r] &0 
}
$$
Then there are no homomorphism from the quotient module to the proper costandard $\onabla(\mu)$ with $\rho(\mu)<\rho(\lambda)$
and by induction assumption we know that the submodule $\imath_{<\lambda}(\imath_{<\lambda}^{!}(M))$ admits an excellent filtration.
Hence, taking $RHom(\ttt,\onabla(\mu))$ we get an isomorphism between the first extension groups:
$$
Ext^{1}_{\Cat}(M,\onabla(\mu)) = Ext^{1}_{\Cat}(M/{\imath_{<\lambda}(\imath_{<\lambda}^{!}(M))},\onabla(\mu))
$$
Therefore, it remains to show that the quotient module admits an excellent filtration.
However, the quotient module is covered by a module which admits a filtration with successive quotients isomorphic to standard modules:
$$
\xymatrix{
0\ar[r] & K \ar[r] & \ww_{\lambda}(\rr_{\lambda}(M/{\imath_{<\lambda}(\imath_{<\lambda}^{!}(M))})) \ar[r] & M/{\imath_{<\lambda}(\imath_{<\lambda}^{!}(M))} \ar[r] &0 
}
$$
such that the kernel $K$ belongs to the subcategory $\Cat^{<\lambda}$. 
 Taking $RHom(\ttt,\onabla(\mu))$ we get an isomorphism $Hom(K,\onabla(\mu))=Ext^{1}(M,\onabla(\mu))=0$ for all 
$\mu\in\Upsilon$ with $\rho(\mu)<\lambda$. 
Hence, $K=0$, since otherwise there exists
$\mu\in\Upsilon$ with $\rho(\mu)<\lambda$
such that $\dim Hom(K,\onabla(\mu)) >0$.
Thus, the quotient module is isomorphic to the sum of standard modules $\Delta(\lambda)$ with the number of standard subfactors
equals to $\dim Hom(M,\onabla(\lambda))$.
\end{proof}

\begin{proposition}
\label{prop::BGG}
If the category $\famod$ is stratified then
for all pairs of elements $\lambda,\mu\in\Upsilon$ 
the multiplicity of the number of standard 
modules of weight $\mu$ in the projective cover $\Proj(\lambda)$
is equal to the multiplicity of irreducible $L(\lambda)$ in the proper (co)standard.
$$
[\Proj(\lambda):\Delta(\mu)] = [\onabla(\mu):L(\lambda)].
$$
\end{proposition}
\begin{proof}
For all $\lambda\in\Xi$ and all modules $M$ we have an identity 
$$
\dim Hom(\Proj(\lambda),M) = [M:L(\lambda)]
$$
because both sides are additive with respect to short exact sequences.
In particular, we have $[\onabla(\mu):L(\lambda)] = \dim Hom(\Proj(\lambda),\onabla(\mu))$ and thanks to Proposition~\ref{prop::exc::filtr}
the latter dimension is equal to the multiplicity of $\Delta(\mu)$ in $\Proj(\lambda)$.
\end{proof}

\begin{corollary}
 If for a $\Z_{+}$-graded Lie algebra $\La$ the category $\famod(\La)$ is stratified then
each projective cover admits a filtration by standard modules and the following identity on multiplicities holds
$$
[\Proj(\lambda,0):\Delta(\mu,k)] = [\onabla(\mu,k):L(\lambda,0)]
$$
If, in addition the Lie algebra has a graded anti-involution then one has:
$$
[\Proj(\lambda,0):\Delta(\mu,k)] = [\onabla(\mu,k):L(\lambda,0)] = [\oDelta(\mu,0):L(\lambda,k)]
$$
\end{corollary}
\begin{proof}
 Follows from Propositions~\ref{prop::BGG} and~\ref{prop::inv::char}.
\end{proof}

\begin{corollary}
\label{cor::ext::incomparable}
For all pairs $\lambda,\mu\in\Upsilon$ such that $\rho(\lambda)\not\leq \rho(\mu)$ we have vanishing of higher extension groups:
$$Ext^{>0}_{\Cat}(\Delta(\lambda),\Delta(\mu))=0$$
\end{corollary}

One can also define a dual filtration.
We say that a module $M\in\Cat$ admits an \emph{$\onabla$--excellent filtration} if there exists a filtration 
whose successive quotients are isomorphic to proper costandard modules. 
 The same criterion is applied for this kind of filtration
\begin{proposition}
 A module $M\in\Cat$ admits an $\onabla$-excellent filtration if and only if $Ext^{1}(\Delta(\lambda),M)=0$ for all $\lambda\in\Xi$.
In this case the number of subfactors isomorphic to $\onabla(\mu)$ is equal to $\dim Hom(\Delta(\mu),M)$.
\end{proposition}

\begin{definition}
We say that a module $M$ is tilting if it is both $\Delta$-filtered and $\onabla$-filtered.
\end{definition}
For each $\lambda\in\Upsilon$ let $\{\mu_0,\ldots,\mu_k\}$ be the set of all elements in $\Xi$ which are less or equal than $\rho(\lambda)$.
We number them in the linear order refining the partial ordering on $\Xi$.
That is $\mu_i\leq\mu_j$ implies $i<j$, in particular $\mu_k=\rho(\lambda)$.
We set $T_k(\lambda):=\Delta(\lambda)$ and the object $T_{i-1}$ is defined as the natural central extension:
$$
\xymatrix{
0\ar[r] & T_{i}(\lambda) \ar[r] & T_{i-1}(\lambda) \ar[r] & 
\ooplus_{\alpha\in\rho^{-1}(\mu_{i-1})}\Delta(\alpha)^{\dim Ext^{1}(\Delta(\alpha),T_i(\lambda))} \ar[r] & 0
}
$$
\begin{lemma}
 The object $T(\lambda):=T_0(\lambda)$ is tilting and does not depend on the choice of refinement of the partial ordering.
$T(\lambda)$ is a tilting hull of $\Delta(\lambda)$ and each tilting module is a direct sum of modules $T(\lambda)$.
\end{lemma}
\begin{proof}
Module $T(\lambda)$ admits a filtration by $\Delta(\mu_i)$ with $\mu_i\leq \lambda$ by construction.
Moreover, we have vanishing $Ext^{1}(\Delta(\alpha),T_j(\lambda))=0$ for all $\alpha\in\Upsilon$ such that $\rho(\alpha)=\mu_i$ with 
$i\geq j$ from the central extension construction. 
In particular, $Ext^{1}(\Delta(\alpha),T(\lambda))=0$ for all $\alpha$ with $\rho(\alpha)\leq\rho(\lambda)$.
Corollary~\ref{cor::ext::incomparable} implies that there are no $Ext(\Delta(\mu),T(\lambda))$ for 
all $\mu$ with $\rho(\mu)\not\leq\lambda$ 
because $T(\lambda)$ is filtered by $\Delta(\mu_i)$ with $\mu_i\leq\lambda$.
\end{proof}

One can use the theory of tilting modules in order to define an equivalence between the 
derived categories assigned with abelian category $\Cat$ and the category of right modules
over $End(\ooplus_{\lambda\in\Upsilon} T(\lambda))$ denoted by $\Cat^{\dual}$.
One can show that the category $\Cat^{\dual}$ is stratified with respect to the opposite partial ordering and 
one ends with an appropriate notion of Ringel duality.
We refer for detailed discussion to I.\,Losev and to Section 2 of~\cite{Losev_Webster} in particular.

The drawback of this construction is that the general theory of stratified categories does not say anything about tensor product of modules.
In particular, there may exists a graded Lie algebra $\La$ such that the decomposition of the tensor product of a pair of tilting modules over $\La$ 
is not isomorphic to a direct sum of tilting modules.

\section{Examples}
\label{sec::examples}
\subsection{Representations with empty zero weight}
\label{sec::ex::zero}
In this subsection we show several examples of graded Lie algebras where the criterion of Corollary~\ref{cor::famod::strat}
is satisfied by obvious reasons.

\begin{corollary}
\label{cor::zero::weight}
If the quotient subalgebra of zero weight $\La_{>0}^{0}$ of a graded Lie algebra $\La$
is empty and the relative cohomology $H^{\udot}(\La,\La_0;\C)$ is trivial
then the category $\famod(\La)$ is stratified.
\end{corollary}
\begin{proof}
If $\La_{>0}^{0}$ is empty then any quotient of $U(\La_{>0}^{0})$ is isomorphic to $\C$ and the category of 
$A_{\lambda}$-modules is just the category of vector spaces for all $\lambda$.
Standard modules and proper standard modules coincides. The functor $\ww_{\lambda}$ is exact.
Together with the vanishing of higher relative cohomology we get the necessary condition of Corollary~\ref{cor::famod::strat}
and the category $\famod$ is stratified.
\end{proof}
\begin{example}
Let $T(3)$ be an $11$-dimensional Lie algebra with $T(3)_0\cong \sl_3$, $T(3)_1\cong_{\sl_3} \C^{3} =\langle x_1,x_2,x_3\rangle$ is a tautological 
vector representation of $\sl_3$
 and $T(3)_2\cong \Lambda^{2}(\C^{3})\cong (\C^{3})^{*}= \langle x_{12}, x_{23}, x_{13} \rangle$ is its dual representation. 
The commutator $[T(3)_1,T(3)_1] \rightarrow T(3)_2$ is a unique nontrivial $\sl_3$-invariant map:
$$
[x_1,x_2]= x_{12}, \ [x_2,x_3] = x_{23}, \ [x_1,x_3] = x_{13} 
$$

\begin{proposition}
 The category $\famod(T(3))$ is stratified.
\end{proposition}
\begin{proof}
We apply Corollary~\ref{cor::zero::weight}. 
Indeed, there is obviously no vectors of weight zero in neither $\C^{3}$ nor $\Lambda^2\C^{3}$.
The relative homological Chevalley Eilenberg complex looks as follows:
$$
\xymatrix{
0\ar[r] & 
\C\langle
{
\left(
\begin{smallmatrix}
{ x_1\wedge x_2\wedge x_{13}\wedge x_{23} + } \\
{ x_2\wedge x_3\wedge x_{12}\wedge x_{13} + } \\  
{ x_3\wedge x_1\wedge x_{12}\wedge x_{23} }       
\end{smallmatrix}
\right)
}
\rangle
\ar[r]
&
\C \langle
{\begin{array}{c}
{ (x_1\wedge x_2 \wedge x_3),  } \\
{(x_{12}\wedge x_{23} \wedge x_{13}) }
\end{array}}
\rangle
\ar[r]
&
\C
\langle
(x_1\wedge x_{23} - x_2\wedge x_{13} + x_3\wedge x_{12})
\rangle
\ar[r] &
0
}
$$
and is acyclic.
\end{proof}
\end{example}
This example has the following generalization in higher dimensions.
\begin{example}
 Let $T(3n)$ be the $\Z_{+}$-graded Lie superalgebra with
$$
T(3n)_0\cong \sl_{3n}, \ \ T(3n)_{1} \cong \Lambda^{n}\C^{3n}, \ \ T(3n)_2 \cong \Lambda^{2n}\C^{3n}, \ \ T(3n)_{i} = 0 \text{ for } i>2.
$$
The graded component $T(3n)_{1}$ is odd with respect to the Lie suparalgebra structure if $n$ is even. 
If $n$ is odd then $T(3n)$ is a Lie algebra with empty super part.
The unique $\gl_{3n}$ equivariant map 
\begin{equation}
\label{eq::T3n::comm}
\begin{array}{cc}
\text{ for } n \text{ odd }\ & \  \Lambda^{2}(\Lambda^{n}\C^{3n})\rightarrow \Lambda^{2n}\C^{3n}, \\
\text{ for } n \text{ even }\ & \  S^{2}(\Lambda^{n}\C^{3n})\rightarrow \Lambda^{2n}\C^{3n},
\end{array}
\end{equation}
defines all nontrivial commutators in the nilpotent ideal $T(3n)_{>0}$.
\begin{proposition}
 The category $\famod(T(3n))$ is stratified.
\end{proposition}
\begin{proof}
Once again proposition follows from Corollary~\ref{cor::zero::weight}.
Indeed there is obviously no vectors of weight zero neither in $\Lambda^{n}\C^{3n}$ nor in its dual representation $\Lambda^{2n}\C^{3n}$.
The relative Chevalley-Eilenberg complex is the complex of invariants 
$$\left[\ooplus_{i,j>0}\Lambda^{\udot}(\Lambda^{n}\C^{3n})\otimes \Lambda^{\udot}(\Lambda^{2n}\C^{3n})\right]^{\sl_{3n}}$$
with differential given by contraction with $\Lambda^{3n}\C^{3n}$ that maps $\Lambda^{i}\otimes\Lambda^{j} \mapsto \Lambda^{i-2}\otimes\Lambda^{j+1}$.
The direct check that all bigraded components $[\Lambda^{i}\otimes\Lambda^{j}]^{\sl_{3n}}$ are eigen spaces for Laplace operator 
with nonzero eigen value. Consequently, the relative Chevalley-Eilenberg complex is indeed acyclic.
\end{proof}
\end{example}

In order to ensure the reader that modules with empty zero weight is not so rare we add the following Lemma:
\begin{lemma}
Let $V\cong \ooplus_{\lambda\in \PPdom} L(\lambda)^{m_{\lambda}}$ be a finite-dimensional representation over 
a semisimple Lie algebra $\g$. Then 
\\
$\phantom{.....}\bullet$ $V$ has empty zero weight if and only if $m_\lambda=0$ for all $\lambda\in \QQ$;
\\
$\phantom{.....}\bullet$ $V$ is self-dual iff $\forall \lambda$  $m_{\lambda}=m_{-\omega_0(\lambda)}$, 
where $\omega_{0}$ is the longest element in the Weyl group.
\end{lemma}
\begin{proof}
The weights of universal enveloping algebra $U(\La_0)$ belongs to the root lattice $\QQ$.
Thus the highest weight vector $p_\lambda$ belongs to the submodule generated by a vector of zero weight
only if $\lambda\in\QQ$. Moreover,  $\forall \lambda\in\QQ\cup\PPdom$ the dimension of $L(\lambda)^{0}$ is greater than zero.

The highest weight vector of the representation dual to $L(\lambda)$ is equal to minus weight of the lowest weight vector.
Thus $L(\lambda)^{*}\cong L(-\omega_0(\lambda))$.
\end{proof}

Unfortunately, if a graded Lie algebra $\La$ has empty zero weight of radical $\La_{>0}$ and admits an anti-involution 
then the relative Lie algebra homology $H^{\udot}(\La,\La_0;\C)$ are nontrivial and $\famod(\La)$ may not be stratified.

\subsection{Parabolic subalgebras}
\label{sec::ex::parab}
Example~\ref{ex::parabolic} of parabolic subalgebra of a Kac--Moody Lie algebra sounds to be interesting, but unfortunately
in most of the cases the corresponding relative Lie algebra cohomology are nontrivial and thus the theory of stratified categories may not be applied..

Recall, that to a Graph $\Gamma$ we can assign a Kac--Moody Lie algebra $\g_{\Gamma}$ whose simple roots are numbered by vertices of a graph.
Let $\langle e_\alpha,h_\alpha,f_\alpha\rangle$ be an $\sl_2$-triple associated with a simple root $\alpha$.
With a proper subgraph $S\subset\Gamma$ one assigns the corresponding Kac--Moody subalgebra $\g_S$ and a parabolic subalgebra 
$\p_{S\subset\Gamma}$ which is generated by $\{e_\alpha : \alpha\in\Gamma\}$ and $\{f_\alpha : \alpha\in S\}$.
Note that the standard notation is to denote by parabolic subalgebra $\mathsf{p}_{S}$ the one that contains Borel subalgebra and $\g_S$. 
Our parabolic subalgebra $\p_{S}$ is the commutator $[\mathsf{p}_S, \mathsf{p}_{S}]$ of the standard one.

\begin{proposition}
\label{prop::parabolic::homology}
For a Dynkin diagram $\Gamma$ of finite type and its proper subgraph $S\subset\Gamma$
the relative cohomology of a finite-dimensional parabolic subalgebra $\p_{S}\subset\g_{\Gamma}$ 
relative to its semisimple subalgebra $\g_S$ is nontrivial:
$$
H^{>0}(\p_S,\g_S;\C) \neq 0
$$
\end{proposition}
\begin{proof}
We believe that this fact is known for specialists, however, we were not able to find a reference.
Moreover, Proposition~\ref{prop::parabolic::homology} 
produces rather negative result and does not affect the main part of the article, therefore, 
we just mention that proof follows by induction on the number of vertices in $\Gamma\setminus S$
from the two following known facts:
\\
(1) Let $S\subset \Gamma$ be a maximal proper subgraph.
Then, the nilpotent radical $\n_{S}$ of a maximal parabolic subalgebra $\p_S$ is an abelian Lie algebra.
Consequently, 
$$H^{\dim(\n_S)}(\p_S,\g_S;\C)= \left[\Lambda^{\dim(\n_S)}\n_S\right]^{\g_S} = \C\neq 0
$$
\\
(2) Let $S\subset T \subset \Gamma$ be a sequence of properly embedded Dynkin graphs.
Let $\p_{S\subset T}$ be a parabolic subalgebra assigned with embedding $S\subset T$,
respectively $\p_{S\subset \Gamma}$ be a parabolic subalgebra of $\g_{\Gamma}$.
Then the embedding of Lie algebras $\p_{S\subset T}\hookrightarrow\p_{S\subset \Gamma}$
produces the map of relative cohomology:
$$
H^{\udot}(\p_{S\subset \Gamma},\g_S;\C) \rightarrow H^{\udot}(\p_{S\subset T},\g_S;\C)
$$
which is known to be surjective.
\end{proof}

However, if $\Gamma$ is an affine Dynkin diagram and $S$ is the corresponding finite Dynkin diagram we come up 
with the main example of this paper when the theory of stratified categories is applied in full generality.

\subsection{One-dimensional currents $\g\otimes \C[x]$ and applications}
\label{sec::ex::C_x_}
In this section we consider Example~\ref{eq::example::parabolic} from Section~\ref{sec::list::examples}.
The corresponding standard modules $\Delta(\lambda,0)$ 
are known after~\cite{CPweyl,FL} under the name \emph{Global Weyl modules} and are typically denoted in the literature
by $W^{g}(\lambda)$ or simply $W(\lambda)$.
Respectively, the proper standard module $\oDelta(\lambda,0)$ is known under the name \emph{local Weyl module} 
and is typically denoted by $W^{loc}(\lambda)$.
This example was the starting point of the whole theory.
The BGG reciprocity for $\g\otimes\C[x]$ was first proved for $\g=\sl_2$ in~\cite{BCM} and for $\g=\sl_n$ in~\cite{BBCKL}.
We explain below how to prove it in the case of general $\g$ using the knowledge of the comparison between local and global Weyl modules.
\begin{conjecture}
\label{thm::Cx::strat}
The category $\famod(\g\otimes \C[x])$ of graded finitely generated $\g$-locally finite modules
over Lie algebra of currents $\g\otimes \C[x]$ is stratified. 
\end{conjecture}
\begin{remark}
 We have proved in Corollary~\ref{cor::ADE::type} that Conjecture is true for $\g$-simply laced and, moreover, we checked by hands in $B$-case 
that for small weights the statement of Conjecture is true.
\end{remark}

In order to ensure that Conjecture is true in general we will check all conditions of Proposition~\ref{cor::famod::strat}
and, consequently, Conjecture~\ref{conj::gen::strat} implies Conjecture~\ref{thm::Cx::strat}.
Indeed, it was shown by Bek and Nakajima in \cite{BN} and by different methods by Naoi in~\cite{Naoi} that 
for the Lie algebra $\g\otimes\C[x]$ with $\g$ simple the global Weyl module of weight $\lambda=\sum \lambda_i\omega_i$ 
is isomorphic to the tensor product of the local Weyl module of the same weight and the algebra 
$A_\lambda = S^{\lambda_1}(\C[x])\times\ldots\times S^{\lambda_r}(\C[x])$.
It remains to recall that the relative Lie algebra of $\g\otimes\C[x]$ relative to $\g$ are trivial.
This follows from the existence of a parabolic BGG resolution (first introduced in~\cite{Lepowsky})
 of a trivial module over affine Lie algebra $\widehat{\g\otimes \C[[x,x^{-1}]}$
whose components are $\g\otimes\C[x]$-modules induced from the irreducible finite-dimensional $\g$-representation $L(w\cdot 0)$ 
with $w\in W^{aff}/W$.
We have
$$H^{\udot}(\g\otimes\C[x],\g;Ind_{\g}^{\g\otimes\C[x]}L(w\cdot 0)) = H^{\udot}(\g,\g;L(w\cdot 0)) = [L(w\cdot 0)]^{\g} = \begin{cases}
\C, \text{ if } w = Id,\\  0, \text{ otherwise. } \end{cases}
$$
Consequently, the relative cohomology 
$H^{\udot}(\g\otimes\C[x],\g;\C)$ is trivial.

\begin{remark}
 There is also a possibility to discuss a parabolic Ivahori subalgebra for the twisted affine Lie algebras.
Y.Sanderson showed in~\cite{S} that characters of Demazure modules of level one are given by Macdonald polynomials 
and coincides with local Weyl modules. (See also~\cite{FMS}).  Consequently, the BGG reciprocity is satisfied in that case.
\end{remark}

Let us give several applications of the general theory we discussed in Section~\ref{sec::filtrations} to the case of 
the Lie algebra $\g\otimes\C[x]$.
This section is motivated by the representational theoretical description of the identities found in~\cite{Chered:Feigin}.

\begin{corollary}
 The character of the module $\oDelta(\lambda)$ in $\famod(\g\otimes\C[x])$ 
is given by $q$-Hermite polynomial (=specialization of Macdonald polynomial at $t=0$).
The constant term $\langle P_{\lambda},P_\lambda\rangle$ is equal to 
$$
\left(\prod_{i=1}^{rk(\La_{0})}\prod_{j=1}^{\lambda_i}(1-q^{j}) \right)^{-1}
$$
where $\lambda=\sum_{i=1}^{rk(\La_{0})} \lambda_i \omega_i$ with $\{\omega_1,\ldots,\omega_{rk(\La_{0})}\}$ 
being the set of fundamental weights.
\end{corollary}

\begin{conjecture}
\label{conj::integr::rep}
 An integrable $\hat{\g}$-module ${\calM}_{k,\lambda}$ of level $k$ and highest weight $\lambda$ admits a  $\Delta$-excellent filtration. 
\end{conjecture}
This conjecture is motivated by the following observation proved by I.\,Cherednik and B.\,Feigin in~\cite{Chered:Feigin}.
They decomposes $\theta$ function (character of integrable representation) as a sum of Macdonald polynomials with coefficients
given by Formula~\eqref{eq::Verlinde::rule} below.

Consider the parabolic BGG resolution of an integrable representations of level $k$ and weight $\lambda$:
\begin{equation}
\label{eq::BGG::parabolic}
\ldots \to \ooplus_{\omega\in W/W_{S}} \Proj(\omega\cdot \lambda)[l(\omega)]
\to \ldots \to \Proj(s_0\cdot\lambda) \to \Proj(\lambda) \twoheadrightarrow \calM(\lambda,k) \to 0
\end{equation}
where $s_0$ is the reflection with respect to the affine root.
Consequently, for all $\mu$ we have
\begin{multline}
\label{eq::Verlinde::rule}
\dim RHom_{\famod(\hat{\g})}({\calM}_{k,\lambda},\onabla(\mu,l)) = 
\sum_{\omega\in W/W_{S}}(-1)^{l(\omega)} \dim RHom_{\famod}(\Proj(\omega\cdot \lambda),\onabla(\mu,l)) = \\
=\sum_{\omega\in W/W_{S}}(-1)^{l(\omega)} \dim Hom_{\g}(L(\omega\cdot \lambda),\onabla(\mu,l)) = 
\sum_{\omega\in W/W_{S}}(-1)^{l(\omega)} [\onabla(\mu,l): L(\omega\cdot\lambda,deg(w))]
\end{multline}
We expect that the complex 
$$\oplus_{s_i} Hom_{\g}( L(s_i s_0\cdot \lambda), \onabla(\mu)) \to Hom_{\g}( L(s_0\cdot \lambda), \onabla(\mu)) \to Hom_{\g}(L(\lambda),\onabla(\mu))$$
is exact in the middle term for all $\lambda$ and $\mu$ what implies vanishing of $Ext^1(\calM_{k,\lambda},\onabla(\mu))$.
Following Proposition~\ref{prop::exc::filtr} we know that the latter implies the existence of a $\Delta$-excellent filtration on $\calM_{k,\lambda}$.

\subsection{Current Lie algebras}
\label{sec::ex::cur}
The theory of local and global Weyl modules was extensively studied in the last decade in the case of arbitrary commutative graded algebra $A$.
(e.g.~\cite{CFK}).
However, this section is devoted to explain that BGG reciprocity holds for current Lie algebras $\g\otimes A$ 
with super-commutative graded algebra $A$ if and only if $A\cong\C[x]$.
In the next Section~\ref{sec::gen::Mac} we will explain what kind of derived complexes one has to consider for $A$ generic.

\begin{theorem}
\label{thm::g_A::stratified}
For a semisimple Lie algebra $\g$ the category $\famod(\g\otimes A)$ of finitely generated graded modules over
the Lie algebra of currents $\g\otimes A$ is stratified if and only if the nontrivial supercommutative algebra $A$ 
is isomorphic to the polynomial ring in one variable.
\end{theorem}
\begin{proof}
Theorem~\ref{thm::Cx::strat} covers the if statement.
In order to show the only if statement we will explain that 
that for $A\not\simeq\C[x]$ there are nontrivial relative cohomology classes in higher degrees,
thus Condition~\eqref{eq::star::van::H} of Corollary~\ref{cor::famod::strat} is not satisfied.
Indeed, if the super-commutative graded unital algebra $A$ is not isomorphic to $\C[x]$ then 
one of the following possibilities is satisfied for $A$:

$(1)$ There exists a nontrivial odd generator,  

$(2)$ The space of even generators is at least two-dimensional, 

$(3)$  $A\simeq \C[x]/(x^{n})$.
\\
In the case $(1)$ we have a natural embedding 
$$H^{\udot}(BG;\C) = [S^{\udot}\g]^{\g} = [\Lambda^{\udot}(\g^{*}\otimes \xi)]^{\g} \hookrightarrow H^{\udot}(\g\otimes A,\g;\C).$$
where $\xi$ is an odd generator of $A$.
Lemma~\ref{lem::rel::coh} below covers the case $(2)$.
The relative cohomology in the last case $(3)$ are known (\cite{FGT}) and is different from $0$ in higher degrees.
The full description of the cohomology $H^{\udot}(\g\otimes\C[x]/x^{n};\C)$ is known under the name ``the strong Macdonald conjecture''. 
\end{proof}

Since $\g$ is semisimple the Killing form $(,)$ is a nondegenerate invariant pairing on $\g$.
Let $e_i=e^{i}$ be the orthonormal basis of $\g\cong\g^{*}$ with respect to this pairing
and $c_{ijk}:=([e_i,e_j],e_k) = (e_i,[e_j,e_k])$ be the structure constants that are skew-symmetric in all three arguments.
We have 
$$
[e_i,e_j] = c_{ijk}e_k \ \text{ and } \ \delta_{CE}(e^i)= c_{ijk}e^{j}e^{k}.  
$$
\begin{lemma}
\label{lem::rel::coh}
Let $x,y$ be a pair of linearly independent linear maps $A_{+}/(A_{+}^{2}) \rightarrow \C$ 
then the relative cochain
$$
\varphi_{x,y}:= \sum_{i=1}^{\dim \g} e^i x\wedge e^i y
$$
form a nontrivial relative 2-cocyle in $H^{\udot}(\g\otimes A,\g;\C)$.
\end{lemma}
\begin{proof}
The space of $1$-cochain in the relative Chevalley-Eilenberg complex $C^{\udot}(\g\otimes A,\g;\C)$
is isomorphic to 
$Hom_{\g}(\g\otimes A_{+},\C)\cong [\g^{*}]^{\g}\otimes Hom_{\C}(A_{+},\C)$ which is empty because $\g$ is semisimple.
The chain $\varphi_{x,y}$ is obviously $\g$ invariant and it remains to show that it is closed 
under the differential in relative Chevalley-Eilenberg complex:
\begin{multline*}
d_{CE}(\sum_{i} e^i x\wedge e^i y) =\sum_{i} d_{CE}(e^{i}x) \wedge e^{i} y - e^{i}x\wedge d_{CE}(e^{i}y) = \\
\sum_{i,j,k} (c_{ijk} e^{j}\wedge e^{k}x )\wedge e^{i} y - e^{i}x\wedge (c_{ijk} e^{j} \wedge e^{k} y) =
\sum_{i,j,k} (c_{ijk} - c_{jik}) e^{j}\wedge e^{k}x\wedge e^{i} y =0
\end{multline*}
\end{proof}

\subsection{Vector fields}
Unfortunately, not much is known so far for Examples~\ref{ex::Wn},\ref{ex::hamilton} of the list~\ref{sec::list::examples}.
\begin{conjecture}
The category $\famod(L_0(n))$ is stratified. 
\end{conjecture}
However, the relative Lie algebra cohomology of the Hamiltonian Lie algebra ${\cal H}_{n}$ is very complicated.
In particular, when $n$ goes to infinity the corresponding cohomology is known under the name Kontsevich Graph-Complex~\cite{Kontsevich}.

\section{Macdonald polynomials for non-stratified categories}
\label{sec::gen::Mac}
As we have seen in Theorem~\ref{thm::g_A::stratified} there are many interesting examples of graded Lie algebras with anti-involution
such that the corresponding category of graded finitely generated modules is not stratified.
In this case one may not expect that Macdonald polynomials will be the characters of specific modules, however, we may try to pose the question
if one can describe complexes of finitely generated modules whose Euler characteristic of characters is given by Macdonald polynomials.
This section is devoted to give an answer on this question.

All conclusions below are stated for the category $\famod(\La)$ where $\La$ is a graded Lie algebra with anti-involution.
However, all theorems remains to be true in the generality of a highest weight category.

Let $\calD^{-}(\famod)$ be the derived category of bounded from above complexes of modules from $\famod$ with finite dimensional graded components
with respect to the inner grading on the Lie algebra $\La$:
$$
\calD^{-}(\famod):=\left\{M^{\udot}= \ldots \to M^{N-1} \to M^{N}\to 0 : M^{i} =\oplus_{j\in\Z} M^{i}_{j},\ 
\forall j\ \dim (\oplus_{i}M^{i}_{j})<\infty \right\}   
$$
With each complex we can assign its graded Euler characteristic $\chi_{q}(M) =\sum_{k,i}(-1)^{i}q^{k}\dim M_k^{i}$.
Recall that category $\famod$ has enough projectives and, in particular, each module $M\in\famod$ has a projective resolutions $P^{\udot}\in\calD^{-}(\famod)$.
This is a useful tool to describe left derived functors.

Let $\calD^{+}(\famod^{op})$ be the corresponding derived category of bounded from below complexes of $\La$-modules whose inner 
grading is bounded from above and, of course, we assume the total finite-dimensionality of each graded component with respect to the inner grading.
The duality functor $\dual:\famod\rightarrow \famod^{op}$ is extended to the functor between corresponding derived categories:
$$
\dual: \calD^{-}(\famod) \rightarrow \calD^{+}(\famod^{op})
$$
and maps projective resolutions to injective ones.
Our finiteness conditions allows us to extend the Ext-pairing~\ref{eq::ext::paring} from the Grothendieck group to the pairing on 
the derived category.
For $M,N\in\calD^{-}(\famod)$  we set 
\begin{equation}
\label{eq::paring::derived}
 \langle M, N\rangle_{\La}:= \sum_{i,k\in\Z}q^{k}(-1)^{i}(\dim RHom^{i}_{\famod}(M,N\{k\}) = 
\sum_{i}(-1)^{i}\dim_{q^{-1}} H^{i}(\La,\La_{0};Hom_{\C}(M,N))
\end{equation}
Note that for a complex $K\in\calD^{+}(\famod^{op})$ the dimension of a graded component of $H^{\udot}(\La,\La_0;K)$
with respect to the inner grading on $\La$ is finite dimensional.
Therefore, the mentioned above Euler characteristic is a well defined Loran series in $q$.

Let $\calD_{\leq\lambda}^{-}(\famod)$ be the full subcategory of $\calD^{-}(\famod)$
 generated by complexes whose cohomology belongs to $\famod^{\leq\lambda}$
and denote by $\bi_{\lambda}$ the corresponding inclusion functor.
\begin{lemma}
\label{lem::adjoint}
 There exists a left adjoint functor $\bi_{\lambda}^{!}:\calD^{-}(\famod)\rightarrow \calD_{\leq\lambda}^{-}(\famod)$.
\end{lemma}
\begin{proof}
This statement follows from the standard technique of admissible subcategories discovered in~\cite{BK}.
Let ${}^{\perp}\calD_{\leq\lambda}(\famod)$ be the left orthogonal to $\calD_{\leq\lambda}(\famod)$ in $\calD(\famod)$.
That is ${}^{\perp}\calD_{\leq\lambda}$ consists of those complexes $C$ yielding the orthogonality condition 
$RHom(C,B)=0$ for all $B\in \calD_{\leq\lambda}$.
In particular, for all $\mu\in\PPdom$ that is not less than or equal to $\lambda$ the projective cover $\Proj(\mu)$ 
belongs to ${}^{\perp}\calD_{\leq\lambda}$, because 
$$
\forall \mu\not\leq\lambda \ \ RHom_{\La}^{\udot}(\Proj(\mu),B) = Hom_{\La_0}(L(\mu),B) = 0 \text{ if } B\in\calD_{\leq\lambda}.
$$
Recall, that $\imath_{\lambda}^{!}$ is the left adjoint functor to the fully faithful inclusion of abelian categories:
$\imath_{\lambda}: \famod^{\leq\lambda} \rightarrow \famod$.
By definition for any module $M$ the module $\imath_{\lambda}^{!}(M)$ is the quotient by the submodule generated by weighted subspaces whose weights
are not less or equal to $\lambda$.
Therefore, there exists a projective module $Q\in\ {}^{\perp}\calD_{\leq\lambda}^{-}$ that surjects onto this submodule.
In other words, $\forall M\in\famod$ there exists an exact sequence $Q^{0}\stackrel{\pi_0}{\to} M \to \imath_{\lambda}^{!}(M)\to 0$ 
with $Q^0$ being a projective module that belongs to ${}^{\perp}\calD_{\leq\lambda}^{-}$.
Denote by $M^{-1}$ the kernel of the map $\pi_0$ and take the corresponding exact sequence for $M^{-1}$:
$Q^{-1}\stackrel{\pi_1}{\to} M^{-1} \to \imath_{\lambda}^{!}(M^{-1})\to 0$. 
Note that if the inner grading of $M$ is bounded from below by $k_0$ (i.e. $M=\oplus_{i\geq k_0} M_i$) then all graded components 
of $M^{-1}$  with degree less than $k_0+1$ are zero.
Therefore, thanks to grading conditions the iteration of this construction is well defined
and we up with a bounded from above complex $(Q^{\udot},\pi_{\ldot})$ whose graded components are finite dimensional,
$Q^{i}$ are projectives modules in $\famod$ orthogonal to $\calD^{-}_{\leq\lambda}$.  
Consequently, $(Q^{\udot},\pi_{\ldot})$ belongs to ${}^{\perp}\calD_{\leq\lambda}^{-}$,
and the homology of the cone of the morphism $(Q^{\ldot},\pi_{\ldot}) \stackrel{\pi_0}{\to} M$ 
is given by $\imath_{\lambda}^{!}(M^{\udot})$ and, consequently, $cone(\pi_0)$ belongs to $\calD_{\leq\lambda}^{-}$.
The same procedure can be repeated if instead of a module $M$ we will consider a bounded from above complex of modules $M^{\udot}\in\calD(\famod)$,
because the functor $\imath_{\lambda}^{!}$ is an additive and right exact functor.
Thus with any element $X\in\calD(\famod)$ we associate an exact triangle
$Q\stackrel{\pi}{\to} X\to cone(\pi)$ such that $Q\in\ {}^{\perp}\calD_{\leq\lambda}$ and $cone(\pi)\in\calD_{\leq\lambda}$.
Following the standard definitions of algebraic geometry~\cite{BK} the category $\calD_{\leq\lambda}$ is called left admissible 
and the assignment $X\mapsto cone(\pi)$ defines a left adjoint functor $\bi_{\lambda}^{!}$.

Note that if the graded components $M_i$ of a module $M\in\famod(\La)$ are zero $\forall i<0$ then 
the graded components of the $k$-th homology of an object $\bi_{\lambda}^{!}(M)$ are zero $\forall i<k$.
Therefore, each graded component of the homology of a complex $\bi_{\lambda}^{!}(M)$ is finite-dimensional.
\end{proof}
The following conjecture generalizes the condition $(h3)$ of a highest weight category.
\begin{conjecture}
\label{conj::indep::order}
For $\La=\g\otimes\C[x,\xi]$ and a pair of incomparable integral dominant weights $\lambda,\mu\in\PPdom$ one has a natural isomorphism
$$
\bi_{\lambda}(\bi_{\lambda}^{!}(\Proj(\lambda))) \cong \bi_{\lambda,\mu}(\bi_{\lambda,\mu}^{!}(\Proj(\lambda)))
$$
where $\bi_{\lambda,\mu}^{!}$ is a left adjoint functor to the inclusion functor 
$\bi_{\lambda,\mu}: \calD_{\leq \{\lambda,\mu\}}^{-} \rightarrow \calD^{-}$
\end{conjecture}
We avoid this problem using the following choice.
Let us fix a total linear ordering $\leq$ on the set of dominant weights $\PPdom$ which dominates the partial ordering $\leq_{\La}$.

\begin{theorem}
\label{thm::Mac::Der::gen}
The set of images $\bi_{\lambda}^{!} \Proj(\lambda)$ of projective covers of irreducibles $L(\lambda)=L(\lambda,0)$ 
form an orthogonal basis with respect to $Ext$-pairing~\eqref{eq::paring::derived}:
$$
Ext^{\udot}_{\La\ttt mod}(\bi_{\lambda}(\bi_{\lambda}^{!}( \Proj(\lambda)))\{k\}, (\bi_{\mu}(\bi_{\mu}^{!} \Proj(\mu)))^{\dual}) = 0, 
\text{ if } \lambda\neq\mu\ \& \ k\geq 0.
$$
The graded Euler characteristic $\chi_{q\ttt\La_0}(\bi_{\lambda}^{!} \Proj(\lambda))$ is well defined and is equal to the dual Macdonald polynomial.
That is:
\begin{equation}
\label{eq::gen::const::term}
\chi_{q\ttt\La_0}(\bi_{\lambda}^{!} \Proj(\lambda)) = z_{\lambda}(q) P_{\lambda}(e^{\alpha},\alpha\in \PP, q),
\quad \text{ where } z_{\lambda}(q)^{-1} := \langle P_{\lambda}, P_{\lambda}\rangle_{\La}
\end{equation}
\end{theorem}
\begin{proof}
Recall, that the duality functor $\dual$ maps projective objects to injectives and, in particular, $\Proj(\mu)^{\dual}$
is an injective hull of $L(\mu)$. Moreover, the functor $(\bi_{\mu}^{!})^{\dual}: \calD(\famod^{op}) \rightarrow \calD_{\leq\lambda}(\famod^{op})$
is the right adjoint to the embedding functor $\calD_{\leq\lambda}^{+}(\famod^{op}) \hookrightarrow \calD^{+}(\famod^{op})$.

If $\lambda > \mu$ we have $\bi_{\lambda} \circ \bi_{\mu} = \bi_{\lambda}$ and
\begin{multline}
\label{eq::ext::comp}
Ext^{\udot}_{\La\ttt mod}\left(\bi_{\lambda}(\bi_{\lambda}^{!} \Proj(\lambda))\{k\}, \bi_{\mu}(\bi_{\mu}^{!} \Proj(\mu))^{\dual}\right)=
Ext^{\udot}_{\La\ttt mod}\left(\bi_{\lambda}(\bi_{\lambda}^{!} \Proj(\lambda))\{k\}, \bi_{\lambda}((\bi_{\mu}^{!})^{\dual} \Proj(\mu)^{\dual})\right) =
\\
=Ext^{\udot}_{\leq\lambda}\left(\bi_{\lambda}^{!}\Proj(\lambda)\{k\},(\bi_{\mu}^{!})^{\dual} \Proj(\mu)^{\dual}\right) =
Ext^{\udot}_{\La\ttt mod}\left(\Proj(\lambda,k),\bi_{\lambda}(\bi_{\mu}^{!} \Proj(\mu))^{\dual}\right) = \\
=Hom_{graded-\La_{0}}\left(L(\lambda,k),\bi_{\mu}((\bi_{\mu}^{!})^{\dual} \Proj(\mu)^{\dual})\right) =0.
\end{multline}
If $\lambda <\mu$ then the vanishing of derived homomorphisms follows from the injectivity of $\Proj(\mu)^{\dual}$.
We avoid the situation that $\lambda$ and $\mu$ are incomparable by a choice of a total linear ordering on the weights.

The cohomology of $\bi_{\lambda}^{!} \Proj(\lambda)$ belongs to the category $\famod^{\leq\lambda}$ 
that is the multiplicity of $L(\mu)$ in $\bi_{\lambda}^{!} \Proj(\lambda)$ may differ from zero only if $\mu\leq\lambda$.
Denote by $z_{\lambda}(q)$ the graded Euler characteristic of the multiplicity of $L(\lambda)$ in $\bi_{\lambda}^{!} \Proj(\lambda)$.
One repeats Computation~\eqref{eq::ext::comp} for $\lambda=\mu$ and get the following:
$$
\chi_{q^{-1}}\left(Ext^{\udot}_{\La\ttt mod}\left(\bi_{\lambda}^{!} \Proj(\lambda), (\bi_{\lambda}^{!} \Proj(\lambda))^{\dual}\right)\right)=
\chi_{q^{-1}}\left(Hom_{\La_0}\left(L(\lambda),(\bi_{\lambda}^{!} \Proj(\lambda))^{\dual}\right)\right) = 
z_{\lambda}(q)
$$
First, one have to notice, that $z_{\lambda}(q)$ is a power series in $q$ and $z_{\lambda}(0)=1$.
Consequently, the polynomials 
$$f_\lambda:=\frac{\chi_{q\ttt\La_0}(\bi_{\lambda}^{!} \Proj(\lambda))}{z_{\lambda}(q)}$$ are well defined and 
satisfy both properties of generalized Macdonald polynomials from Definition~\ref{def::poly::Macdonald} and, thus, have to coincide with them.
Moreover,
$$
\langle f_\lambda, f_\lambda \rangle_{\La} = 
\langle  \frac{\chi_{q\ttt\La_0}(\bi_{\lambda}^{!} \Proj(\lambda))}{z_{\lambda}(q)},
\frac{\chi_{q\ttt\La_0}(\bi_{\lambda}^{!} \Proj(\lambda))}{z_{\lambda}(q)} \rangle_{\La} = 
\frac{\chi_{q^{-1}}\left(Ext^{\udot}_{\La\ttt mod}\left(\bi_{\lambda}^{!} \Proj(\lambda), (\bi_{\lambda}^{!} \Proj(\lambda))^{\dual}\right)\right)}
{z_{\lambda}(q)^{2}}=
\frac{z_{\lambda}(q)}{z_{\lambda}(q)^{2}} = \frac{1}{z_{\lambda}(q)}.
$$
\end{proof}

\begin{remark}
 The category $\famod$ is stratified if and only if the derived category $\calD_{\leq\lambda}(\famod)$ is isomorphic to the derived 
category $\calD(\famod^{\leq\lambda})$.
Therefore, if the category $\famod$ is stratified the left adjoint functor $\bi_{\lambda}^{!}$ between derived categories
becomes the left derived of the corresponding adjoint between abelian categories:
$$
\bi_{\lambda}^{!} = L^{\udot}\imath_{\lambda} : \calD(\famod) \rightarrow \calD(\famod^{\leq\lambda}). 
$$
In particular, since $\Proj(\lambda)$ is projective its image does not have higher derived images:
$$\bi_{\lambda}^{!}(\Proj(\lambda))= L^{\udot}\imath_{\lambda}(\Proj(\lambda)) = L^{0}\imath_{\lambda}(\Proj(\lambda)) = \Delta(\lambda).$$
This conclusion agrees with the statement of Theorem~\ref{thm::Mac::Delta}.
\end{remark}


\begin{thebibliography}{80}

\bibitem{BC} M.~Bennett, V.~Chari,
\emph{Tilting modules for the current algebra of a simple Lie algebra.}
\texttt{math.arXiv}:1202.6050

\bibitem{BBCKL} M.~Bennett, A.~Berenshtein, V.~Chari, A.~Khoroshkin, S.~Loktev
\emph{Macdonald Polynomials and BGG reciprocity for current algebras.} \texttt{math.arXiv}:1207.2446, 
{\it to appear} in Selecta Math.


\bibitem{BCM} M.~Bennett, V.~Chari, and N.~Manning \emph{BGG reciprocity for  current algebras}, Adv. Math. 231 (2012), no. 1, 276-305

\bibitem{BGG}
{J.\,N.\,Bernstein; I.\,M.\,Gelfand; S.\,I.\,Gelfand.}
{\em Structure of representations that are generated by vectors of highest weight. }
(Russian) Func.Anal \& Appl. 5 (1971), no. 1, 1--9.


\bibitem{BN} J.\,Beck, H.\,Nakajima. 
\emph{Crystal bases and two-sided cells of quantum affine algebras.}
 Duke Math. J., 123(2):335–402, 2004.

\bibitem{BK} A.\,I.\,Bondal, M.\,M.\,Kapranov, 
\emph{Representable functors, Serre functors, and mutations}, 
Izv. Akad. Nauk SSSR Ser. Mat., 53:6 (1989), 1183--1205 

\bibitem{CFK} V.~Chari, G.~Fourier and T.~Khandai, \emph{A categorical approach to Weyl modules,} Transform. Groups \textbf{15} (2010), no. 3,  517--549.

\bibitem{CG} V.~Chari and J.~Greenstein, \emph{Current algebras, highest weight categories and quivers,} Adv. Math. \textbf{216} (2007), no. 2, 811--840.

\bibitem{CI} V.\,Chari and B.\,Ion, \emph{BGG reciprocity for current algebras}
\texttt{math.arXiv}:1307.1440 

\bibitem{CL} V.~Chari and S.~Loktev, \emph{Weyl, Demazure and fusion modules for the current algebra of ${\mathsf{sl}}_{r+1}$.} Adv. Math. \textbf{207} (2006), 928--960.

\bibitem{CPweyl} V.~Chari and A.~Pressley, \emph{Weyl modules for classical and quantum affine algebras,} Represent. Theory \textbf{5} (2001), 191--223 (electronic).

\bibitem{Chered:Feigin} I.~Cherednik, B.~Feigin \emph{Rogers-Ramanujan type identities and Nil-DAHA.} \texttt{math.arXiv}:1209.1978 

\bibitem{CPS} E.~Cline, B.~Parshall and L.~Scott, \emph{Finite dimensional algebras and highest weight categories,} J. Reine Angew. Math. \textbf{391} (1988), 85-99.

\bibitem{Donkin} S.~Donkin, \emph{Tilting modules for algebraic groups and finite dimensional algebras} A handbook of tilting theory, London Math. Society Lect. Notes 332 (2007), 215--257.

\bibitem{Fe1} B.\,L.\,Feigin;
\emph{On the cohomology of the Lie algebra of vector fields and of the current algebra.}
Sel.\,Math.\,Sov. 7 (1988), 49--62.


\bibitem{FL} B.\,Feigin, S.\,Loktev;
\emph{Multi-dimensional Weyl modules and symmetric functions.}
 Comm.\,Math.\,Phys 251(3):427--445, 2004.

\bibitem{FMS}
G.\,Fourier, N.\,Manning, A.\,Savage;
\emph{Global Weyl modules for equivariant map algebras}
\texttt{math.arXiv}:1303.4437

\bibitem{FGT}
S.\,Fishel,I.\,Grojnowski, C.\,Teleman; \emph{The strong Macdonald conjecture and Hodge theory on the loop Grassmannian.}
 Ann. of Math. (2), 168 (1) (2008), pp. 175--220.

\bibitem{FoL} G.~Fourier and P.~Littelmann, \emph{Weyl modules, Demazure modules, KR-modules, crystals, fusion products and limit constructions}, Adv. Math. \textbf{211} (2007), no. 2, 566--593.

\bibitem{Haiman} M.\,Haiman.
\emph{Hilbert schemes, polygraphs and the Macdonald positivity conjecture.}
J.\,Amer.\,Math.\,Soc. 14 (2001), no. 4, 941--1006

\bibitem{Humpreys} J.~Humphreys, 
\emph{Representations of semisimple Lie algebras in the BGG category O.}
Graduate Studies in Mathematics, 94. Am.\,Math.\,Soc., P., RI, 2008. xvi+289 pp.

\bibitem{Ion} B.~Ion, {\em Nonsymmetric Macdonald polynomials and Demazure characters}, Duke Math. J {\bf 116}(2) 2003, 299--318

\bibitem{Irving} R.~Irving, {\em BGG algebras and the BGG reciprocity principle.} J. Algebra 135 (1990), no. 2, 363--380.

\bibitem{Jacobson} N.\,Jacobson, \emph{ Lie algebras.} No. 10. Courier Dover Publications, 1979.

\bibitem{Joseph} A.\,Joseph,
\emph{Modules with a Demazure flag.} Studies in Lie theory, 131--169,
Progr. Math., 243, Birkh\"auser Boston, 2006.

\bibitem{Kontsevich} M.\,Kontsevich.
\emph{Feynman diagrams and low-dimensional topology.} 
Progr. Math., 120:97--121, 1994. First European Congress of Mathematics, Vol. II, (Paris, 1992).

\bibitem{Levi}
E.\,E.\,Levi,  \emph{Sulla struttura dei gruppi finiti e continui.} 
Atti della Reale Accademia delle Scienze di Torino. (in Italian) 1905

\bibitem{Lepowsky}
{ J.\,Lepowsky,}
\emph{ A generalization of the Bernstein-Gelfand-Gelfand resolution.}
J. Algebra 49 (1977), no. 2, 496--511.


\bibitem{Losev_Webster}
I.\,Losev, B.\,Webster
\emph{On uniqueness of tensor products of irreducible categorifications}
\texttt{math.arXiv}:1303.1336

\bibitem{Mac} I.~G.~McDonald, {\em Symmetric Functions and Hall Polynomials}, Oxford Univ. Press.

\bibitem{Macdonald1} I.~G.~McDonald,
\emph{Symmetric functions and orthogonal polynomials.}
 University Lecture Series, 12. American Mathematical Society, Providence, RI, 1998.

\bibitem{Manin} Ju.\,I.\,Manin,
{\em Lectures on the K-functor in algebraic geometry.} 
 Uspehi Mat. Nauk 24 1969 no. 5 (149), 3--86.

\bibitem{Naoi} K.\,Naoi;
\emph{Demazure modules and graded limits of minimal affinizations}
\texttt{math.arXiv}:1210.0175



\bibitem{S} Y.\,Sanderson, {\em On the Connection Between Macdonald Polynomials and Demazure Characters}, J. Algebraic Comb. {\bf 11} (2000), 269--275.

\bibitem{Weibel}
{ C.\,Weibel. }
{\em An introduction to homological algebra. }
Cambridge Studies in Advanced Mathematics, 38. 
Cambridge University Press, Cambridge, 1994. 450 pp.

\end{thebibliography}
\end{document}